\documentclass[11pt]{amsart}
\usepackage{amssymb,mathrsfs}
\newtheorem{theorem}{Theorem}[section] 
\newtheorem{lemma}[theorem]{Lemma} 
\newtheorem{proposition}[theorem]{Proposition} 
\newtheorem{corollary}[theorem]{Corollary} 
\theoremstyle{definition}

\begin{document}

\title{A metric on $S^2 \times S^2$ with positive sectional curvature}
\author{Simon Brendle and Pei-Ken Hung}
\address{Columbia University \\ 2990 Broadway \\ New York NY 10027 \\ USA}
\address{University of Illinois \\ 1409 W. Green Street \\ Urbana IL 61801 \\ USA}
\thanks{The first author was supported by the National Science Foundation under grant DMS-2403981 and by the Simons Foundation. He acknowledges the hospitality of T\"ubingen University, where part of this work was carried out. The second author was partially supported by the INSURE Program at the University of Illinois Urbana-Champaign. A number of mathematicians have provided valuable comments on an earlier version of this paper. We thank all of them.}
\begin{abstract}
We construct a metric on $S^2 \times S^2$ with positive sectional curvature. Starting from the standard metric on $S^2 \times S^2$, we first perform a Cheeger deformation. The resulting metric has nonnegative sectional curvature. We refer to it as a Cheeger-M\"uter metric. We then consider a suitable third order perturbation of this Cheeger-M\"uter metric and show that the perturbed metrics have positive sectional curvature.

The proof requires various calculations, some of which have been carried out with the aid of \textsc{Mathematica}. The \textsc{Mathematica} code is attached to this submission.
\end{abstract}
\maketitle 

\section{Introduction} 

A central topic in differential geometry is to understand the interplay between curvature and topology of Riemannian manifolds. A natural question is which manifolds admit metrics of positive sectional curvature. The three-dimensional case is well understood by the famous work of Hamilton \cite{Hamilton}. Hamilton showed that a three-manifold that admits a metric of positive sectional curvature is diffeomorphic to a quotient of $S^3$ by standard isometries. Hsiang and Kleiner \cite{Hsiang-Kleiner} studied four-manifolds with positive sectional curvature and continuous symmetries. They showed that an orientable four-manifold with positive sectional curvature that admits a non-trivial Killing vector field must be homeomorphic to $S^4$ or $\mathbb{CP}^2$. In particular, a metric on $S^2 \times S^2$ with positive sectional curvature cannot admit a non-trivial Killing vector field.

The following is the main result of this paper.

\begin{theorem}
\label{main.theorem}
There exists a metric on $S^2 \times S^2$ with positive sectional curvature. 
\end{theorem}

In a fundamental paper, Cheeger \cite{Cheeger} introduced new tools for constructing metrics with nonnegative sectional curvature via group actions. In particular, for each $t > 0$, Cheeger's construction gives a metric $g_{\text{\rm CM},t}$ on $S^2 \times S^2$ with nonnegative sectional curvature (see also \cite{Bourguignon}). The resulting metrics were studied further by M\"uter \cite{Mueter}. We refer to them as Cheeger-M\"uter metrics. 

To fix notation, we put $M = S^2 \times S^2$. We denote by 
\[\Delta_+ = \{(p_1,p_2) \in S^2 \times S^2: p_1 = p_2\}\] 
the diagonal and by 
\[\Delta_- = \{(p_1,p_2) \in S^2 \times S^2: p_1 = -p_2\}\] 
the anti-diagonal. Moreover, we define 
\[M_{\text{\rm nondeg}} = M \setminus (\Delta_+ \cup \Delta_-) = \{(p_1,p_2) \in S^2 \times S^2: p_1 \neq \pm p_2\}.\] 
The Cheeger-M\"uter metric $g_{\text{\rm CM},t}$ has the important property that, at each point in $M_{\text{\rm nondeg}}$, there is exactly one two-plane with zero sectional curvature, and all other two-planes have strictly positive sectional curvature.

In the next step, we consider the metric $g = \frac{1}{2} \, g_{\text{\rm CM},1}$. We construct a third order perturbation of the metric of the form 
\[\tilde{g}_s = g + s \, h^{(1)} + s^2 \, h^{(2)} + s^3 \, h^{(3)},\] 
where $h^{(1)}, h^{(2)}, h^{(3)}$ are suitably chosen symmetric $(0,2)$-tensors on $M$ and $s$ is a real number close to $0$. 

Let us fix a point in $M_{\text{\rm nondeg}}$. For each $s$, we consider the minimum sectional curvature of the metric $\tilde{g}_s$ at the given point, where the minimum is taken over all two-planes at the given point. We show that the first variation of the minimum sectional curvature is equal to $0$ (see Proposition \ref{first.order.change.in.sectional.curvature} below). Moreover, we show that the second variation of the minimum sectional curvature is nonnegative everywhere and is strictly positive away from a special 2D torus (see Proposition \ref{second.order.change.in.sectional.curvature}). Finally, we show that, along this special torus, the third variation of the minimum sectional curvature is equal to a non-zero constant $\mu$ (see Proposition \ref{third.order.change.in.sectional.curvature.pointwise.version}). 

We now distinguish two cases, depending on the sign of $\mu$. If $\mu$ is positive, then the minimum sectional curvature of the metric $\tilde{g}_s$ at any given point in $M_{\text{\rm nondeg}}$ is strictly positive if $s > 0$ is sufficiently close to $0$. Similarly, if $\mu$ is negative, then the minimum sectional curvature of the metric $\tilde{g}_s$ at any given point in $M_{\text{\rm nondeg}}$ is strictly positive if $s < 0$ is sufficiently close to $0$. In order to analyze the behavior near $\Delta_+$ and $\Delta_-$, we argue by continuity, exploiting the fact that $M_{\text{\rm nondeg}}$ is a dense subset of $M$. In this step, we use Theorem \ref{positivity.manifold.version}.

Finally, we note that perturbations of product metrics were studied in the work of Bourguignon, Deschamps, and Sentenac \cite{Bourguignon-Deschamps-Sentenac}.

\section{An abstract framework}\label{framework}

Let $m$ and $n$ be positive integers. For each $r>0$, we denote by $B_r^m$ the open ball of radius $r$ in $\mathbb{R}^m$. Similarly, we denote by $B_r^n$ the open ball of radius $r$ in $\mathbb{R}^n$. 

Let $r_0$ and $s_0$ be positive real numbers, and let $u: \bar{B}_{r_0}^m \times \bar{B}_{r_0}^n \times [-s_0,s_0] \to \mathbb{R}$ be a smooth function. We assume that $u^{(0)},u^{(1)},u^{(2)},u^{(3)}: \bar{B}_{r_0}^m \times \bar{B}_{r_0}^n \to \mathbb{R}$ are smooth functions such that 
\[u(x,w,s) = u^{(0)}(x,w) + s \, u^{(1)}(x,w) + s^2 \, u^{(2)}(x,w) + s^3 \, u^{(3)}(x,w) + O(s^4)\] 
for $x \in \bar{B}_{r_0}^m$, $w \in \bar{B}_{r_0}^n$, and $s \in [-s_0,s_0]$. We assume that the function $u^{(0)}$ satisfies $u^{(0)}(x,0) = 0$ and $u^{(0)}(x,w) \geq \sigma \, ((w^1)^2+\hdots+(w^n)^2)$ for all $x \in \bar{B}_{r_0}^m$ and $w \in \bar{B}_{r_0}^n$, where $\sigma$ is a positive constant. In particular, for each point $x \in \bar{B}_{r_0}^m$, we have $u^{(0)}_{w^\alpha}(x,0) = 0$ and the matrix $(u^{(0)}_{w^\alpha w^\beta}(x,0))_{1 \leq \alpha,\beta \leq n}$ is positive definite. Finally, we assume that the function $u^{(1)}$ satisfies $u^{(1)}(x,0) = 0$ for all $x \in \bar{B}_{r_0}^m$.

For $x \in \bar{B}_{r_0}^m$ and $s \in [-s_0,s_0]$, we define 
\[U(x,s) = \inf_{w \in \bar{B}_{r_0}^n} u(x,w,s).\] 
The function $U$ is smooth if $|s|$ is sufficiently small. 

For each point $x \in \bar{B}_{r_0}^m$, we define $\bar{w}(x) \in \mathbb{R}^n$ by 
\begin{equation} 
\label{definition.of.bar.w}
\sum_{\beta=1}^n u^{(0)}_{w^\alpha w^\beta}(x,0) \, \bar{w}^\beta(x) = -u^{(1)}_{w^\alpha}(x,0). 
\end{equation}
Moreover, for each point $x \in \bar{B}_{r_0}^m$, we define 
\[U^{(2)}(x) = u^{(2)}(x,0) + \frac{1}{2} \sum_{\alpha=1}^n u^{(1)}_{w^\alpha}(x,0) \, \bar{w}^\alpha(x)\] 
and 
\begin{align*} 
U^{(3)}(x) 
&= u^{(3)}(x,0) + \sum_{\alpha=1}^n u^{(2)}_{w^\alpha}(x,0) \, \bar{w}^\alpha(x) + \frac{1}{2} \sum_{\alpha,\beta=1}^n u^{(1)}_{w^\alpha w^\beta}(x,0) \, \bar{w}^\alpha(x) \, \bar{w}^\beta(x) \\ 
&+ \frac{1}{6} \sum_{\alpha,\beta,\gamma=1}^n u^{(0)}_{w^\alpha w^\beta w^\gamma}(x,0) \, \bar{w}^\alpha(x) \, \bar{w}^\beta(x) \, \bar{w}^\gamma(x). 
\end{align*} 

\begin{lemma}
\label{asymptotic.expansion.of.U}
The function $U$ satisfies the asymptotic expansion 
\[U(x,s) = s^2 \, U^{(2)}(x) + s^3 \, U^{(3)}(x) + O(s^4)\] 
for $|s|$ sufficiently small.
\end{lemma}

\textbf{Proof.}
Throughout the proof, we assume that $|s|$ is sufficiently small. By the implicit function theorem, we can find a smooth function $w^*(x,s)$ such that 
\[u_{w^\alpha}(x,w,s) \big |_{w=w^*(x,s)} = 0\] 
and 
\[w^*(x,s) = s \, \bar{w}(x) + s^2 \, \hat{w}(x) + O(s^3),\] 
where $\hat{w}(x) \in \mathbb{R}^n$ is defined by 
\begin{align*} 
\sum_{\beta=1}^n u^{(0)}_{w^\alpha w^\beta}(x,0) \, \hat{w}^\beta(x) 
&= -u^{(2)}_{w^\alpha}(x,0) - \sum_{\beta=1}^n u^{(1)}_{w^\alpha w^\beta}(x,0) \, \bar{w}^\beta(x) \\ 
&- \frac{1}{2} \sum_{\beta,\gamma=1}^n u^{(0)}_{w^\alpha w^\beta w^\gamma}(x,0) \, \bar{w}^\beta(x) \, \bar{w}^\gamma(x). 
\end{align*} 
With this understood, 
\[U(x,s) = \inf_{w \in \bar{B}_{r_0}^n} u(x,w,s) = u(x,w^*(x,s),s)\] 
if $|s|$ is sufficiently small. Moreover, 
\begin{align*} 
&u(x,w^*(x,s),s) \\ 
&= s^2 \, u^{(2)}(x,0) + s^2 \sum_{\alpha=1}^n u^{(1)}_{w^\alpha}(x,0) \, \bar{w}^\alpha(x) + \frac{1}{2} \, s^2 \sum_{\alpha,\beta=1}^n u^{(0)}_{w^\alpha w^\beta}(x,0) \, \bar{w}^\alpha(x) \, \bar{w}^\beta(x) \\ 
&+ s^3 \, u^{(3)}(x,0) + s^3 \sum_{\alpha=1}^n u^{(2)}_{w^\alpha}(x,0) \, \bar{w}^\alpha(x) + s^3 \sum_{\alpha=1}^n u^{(1)}_{w^\alpha}(x,0) \, \hat{w}^\alpha(x) \\ 
&+ \frac{1}{2} \, s^3 \sum_{\alpha,\beta=1}^n u^{(1)}_{w^\alpha w^\beta}(x,0) \, \bar{w}^\alpha(x) \, \bar{w}^\beta(x) + s^3 \sum_{\alpha,\beta=1}^n u^{(0)}_{w^\alpha w^\beta}(x,0) \, \bar{w}^\alpha(x) \, \hat{w}^\beta(x) \\ 
&+ \frac{1}{6} \, s^3 \sum_{\alpha,\beta,\gamma=1}^n u^{(0)}_{w^\alpha w^\beta w^\gamma}(x,0) \, \bar{w}^\alpha(x) \, \bar{w}^\beta(x) \, \bar{w}^\gamma(x) + O(s^4). 
\end{align*}
Using (\ref{definition.of.bar.w}), we obtain 
\begin{align*} 
&u(x,w^*(x,s),s) \\ 
&= s^2 \, u^{(2)}(x,0) + \frac{1}{2} \, s^2 \sum_{\alpha=1}^n u^{(1)}_{w^\alpha}(x,0) \, \bar{w}^\alpha(x) \\ 
&+ s^3 \, u^{(3)}(x,0) + s^3 \sum_{\alpha=1}^n u^{(2)}_{w^\alpha}(x,0) \, \bar{w}^\alpha(x) + \frac{1}{2} \, s^3 \sum_{\alpha,\beta=1}^n u^{(1)}_{w^\alpha w^\beta}(x,0) \, \bar{w}^\alpha(x) \, \bar{w}^\beta(x) \\ 
&+ \frac{1}{6} \, s^3 \sum_{\alpha,\beta,\gamma=1}^n u^{(0)}_{w^\alpha w^\beta w^\gamma}(x,0) \, \bar{w}^\alpha(x) \, \bar{w}^\beta(x) \, \bar{w}^\gamma(x) + O(s^4). 
\end{align*}
Putting these facts together, the assertion follows. This completes the proof of Lemma \ref{asymptotic.expansion.of.U}. \\

\begin{lemma}
\label{derivative.of.U2}
We have 
\begin{align*}
U^{(2)}_{x^i}(x) 
&= u^{(2)}_{x^i}(x,0) + \sum_{\alpha=1}^n u^{(1)}_{x^i w^\alpha}(x,0) \, \bar{w}^\alpha(x) \\ 
&+ \frac{1}{2} \sum_{\alpha,\beta=1}^n u^{(0)}_{x^i w^\alpha w^\beta}(x,0) \, \bar{w}^\alpha(x) \, \bar{w}^\beta(x). 
\end{align*}
\end{lemma}

\textbf{Proof.}
Differentiating the identity (\ref{definition.of.bar.w}) with respect to $x^i$ gives 
\begin{equation} 
\label{derivative.of.bar.w}
\sum_{\beta=1}^n u^{(0)}_{w^\alpha w^\beta}(x,0) \, \partial_{x^i} \bar{w}^\beta(x) = -u^{(1)}_{x^i w^\alpha}(x,0) - \sum_{\beta=1}^n u^{(0)}_{x^i w^\alpha w^\beta}(x,0) \, \bar{w}^\beta(x). 
\end{equation}
Using (\ref{definition.of.bar.w}) and (\ref{derivative.of.bar.w}), we obtain 
\begin{align*}
U^{(2)}_{x^i}(x) 
&= u^{(2)}_{x^i}(x,0) + \frac{1}{2} \sum_{\alpha=1}^n u^{(1)}_{x^i w^\alpha}(x,0) \, \bar{w}^\alpha(x) \\ 
&+ \frac{1}{2} \sum_{\alpha=1}^n u^{(1)}_{w^\alpha}(x,0) \, \partial_{x^i} \bar{w}^\alpha(x) \\ 
&= u^{(2)}_{x^i}(x,0) + \frac{1}{2} \sum_{\alpha=1}^n u^{(1)}_{x^i w^\alpha}(x,0) \, \bar{w}^\alpha(x) \\ 
&- \frac{1}{2} \sum_{\alpha,\beta=1}^n u^{(0)}_{w^\alpha w^\beta}(x,0) \, \bar{w}^\beta(x) \, \partial_{x^i} \bar{w}^\alpha(x) \\ 
&= u^{(2)}_{x^i}(x,0) + \sum_{\alpha=1}^n u^{(1)}_{x^i w^\alpha}(x,0) \, \bar{w}^\alpha(x) \\ 
&+ \frac{1}{2} \sum_{\alpha,\beta=1}^n u^{(0)}_{x^i w^\alpha w^\beta}(x,0) \, \bar{w}^\alpha(x) \, \bar{w}^\beta(x). 
\end{align*}
This completes the proof of Lemma \ref{derivative.of.U2}. \\

In the next step, we fix an open set $\Omega \subset B_{r_0}^m$ and a positive real number $r_1 \leq r_0$. Suppose that $\Upsilon: \Omega \times B_{r_1}^n \to \mathbb{R}$ is a positive smooth function. Suppose further that $\varphi: \Omega \times B_{r_1}^n \to B_{r_0}^m$ and $\psi: \Omega \times B_{r_1}^n \to B_{r_0}^n$ are smooth maps such that $\varphi(y,0) = y$ and $\psi(y,0) = 0$ for each $y \in \Omega$. Moreover, we assume that the matrix $(\psi^\beta_{z^\alpha}(y,0))_{1 \leq \alpha,\beta \leq n}$ is invertible for each point $y \in \Omega$. 

We define smooth functions $v^{(0)},v^{(1)},v^{(2)},v^{(3)}: \Omega \times B_{r_1}^n \to \mathbb{R}$ by 
\begin{align*} 
v^{(0)}(y,z) &= \Upsilon(y,z) \, u^{(0)}(\varphi(y,z),\psi(y,z)), \\ 
v^{(1)}(y,z) &= \Upsilon(y,z) \, u^{(1)}(\varphi(y,z),\psi(y,z)), \\ 
v^{(2)}(y,z) &= \Upsilon(y,z) \, u^{(2)}(\varphi(y,z),\psi(y,z)), \\ 
v^{(3)}(y,z) &= \Upsilon(y,z) \, u^{(3)}(\varphi(y,z),\psi(y,z)) 
\end{align*} 
for $y \in \Omega$ and $z \in B_{r_1}^n$. 

\begin{lemma}
\label{v0}
The function $v^{(0)}$ satisfies $v^{(0)}(y,0) = 0$, $v^{(0)}_{z^\alpha}(y,0) = 0$, and 
\[v^{(0)}_{z^\alpha z^\beta}(y,0) = \sum_{\gamma,\delta=1}^n \Upsilon(y,0) \, u^{(0)}_{w^\gamma w^\delta}(y,0) \, \psi^\gamma_{z^\alpha}(y,0) \, \psi^\delta_{z^\beta}(y,0)\] 
for each point $y \in \Omega$. Moreover, for each point $y \in \Omega$, we have 
\begin{align*} 
v^{(0)}_{z^\alpha z^\beta z^\gamma}(y,0) 
&= \sum_{\delta,\zeta,\eta=1}^n \Upsilon(y,0) \, u^{(0)}_{w^\delta w^\zeta w^\eta}(y,0) \, \psi^\delta_{z^\alpha}(y,0) \, \psi^\zeta_{z^\beta}(y,0) \, \psi^\eta_{z^\gamma}(y,0) \\ 
&+ T_{\alpha\beta\gamma} + T_{\beta\gamma\alpha} + T_{\gamma\alpha\beta}, 
\end{align*}
where 
\begin{align*} 
T_{\alpha\beta\gamma} 
&= \sum_{\delta,\zeta=1}^n \Upsilon(y,0) \, u^{(0)}_{w^\delta w^\zeta}(y,0) \, \psi^\delta_{z^\gamma}(y,0) \, \psi^\zeta_{z^\alpha z^\beta}(y,0) \\ 
&+ \sum_{i=1}^m \sum_{\delta,\zeta=1}^n \Upsilon(y,0) \, u^{(0)}_{x^i w^\delta w^\zeta}(y,0) \, \varphi^i_{z^\gamma}(y,0) \, \psi^\delta_{z^\alpha}(y,0) \, \psi^\zeta_{z^\beta}(y,0) \\ 
&+ \sum_{\delta,\zeta=1}^n \Upsilon_{z^\gamma}(y,0) \, u^{(0)}_{w^\delta w^\zeta}(y,0) \, \psi^\delta_{z^\alpha}(y,0) \, \psi^\zeta_{z^\beta}(y,0). 
\end{align*}
\end{lemma}

\textbf{Proof.} 
This follows from the chain rule, keeping in mind that $u^{(0)}(x,0) = 0$ and $u^{(0)}_{w^\alpha}(x,0) = 0$ for all $x$. \\

\begin{lemma}
\label{v1}
The function $v^{(1)}$ satisfies $v^{(1)}(y,0) = 0$, 
\[v^{(1)}_{z^\alpha}(y,0) = \sum_{\beta=1}^n \Upsilon(y,0) \, u^{(1)}_{w^\beta}(y,0) \, \psi^\beta_{z^\alpha}(y,0)\] 
and 
\begin{align*} 
v^{(1)}_{z^\alpha z^\beta}(y,0) 
&= \sum_{\gamma,\delta=1}^n \Upsilon(y,0) \, u^{(1)}_{w^\gamma w^\delta}(y,0) \, \psi^\gamma_{z^\alpha}(y,0) \, \psi^\delta_{z^\beta}(y,0) \\ 
&+ \sum_{\gamma=1}^n \Upsilon(y,0) \, u^{(1)}_{w^\gamma}(y,0) \, \psi^\gamma_{z^\alpha z^\beta}(y,0) + S_{\alpha\beta} + S_{\beta\alpha} 
\end{align*}
for each point $y \in \Omega$, where 
\begin{align*} 
S_{\alpha\beta} 
&= \sum_{i=1}^m \sum_{\gamma=1}^n \Upsilon(y,0) \, u^{(1)}_{x^i w^\gamma}(y,0) \, \varphi^i_{z^\alpha}(y,0) \, \psi^\gamma_{z^\beta}(y,0) \\ 
&+ \sum_{\gamma=1}^n \Upsilon_{z^\alpha}(y,0) \, u^{(1)}_{w^\gamma}(y,0) \, \psi^\gamma_{z^\beta}(y,0). 
\end{align*}
\end{lemma}

\textbf{Proof.} 
This follows from the chain rule, keeping in mind that $u^{(1)}(x,0) = 0$ for all $x$. \\

\begin{lemma}
\label{v2}
The function $v^{(2)}$ satisfies 
\[v^{(2)}(y,0) = \Upsilon(y,0) \, u^{(2)}(y,0)\] 
and 
\begin{align*} 
v^{(2)}_{z^\alpha}(y,0) 
&= \sum_{\beta=1}^n \Upsilon(y,0) \, u^{(2)}_{w^\beta}(y,0) \, \psi^\beta_{z^\alpha}(y,0) \\ 
&+ \sum_{i=1}^m \Upsilon(y,0) \, u^{(2)}_{x^i}(y,0) \, \varphi^i_{z^\alpha}(y,0) \\ 
&+ \Upsilon_{z^\alpha}(y,0) \, u^{(2)}(y,0) 
\end{align*}
for each point $y \in \Omega$.
\end{lemma}

\textbf{Proof.}
This follows from the chain rule. \\

It follows from Lemma \ref{v0} that the matrix $(v^{(0)}_{z^\alpha z^\beta}(y,0))_{1 \leq \alpha,\beta \leq n}$ is positive definite for each point $y \in \Omega$. For each point $y \in \Omega$, we define $\bar{z}(y) \in \mathbb{R}^n$ by 
\begin{equation} 
\label{definition.of.bar.z}
\sum_{\beta=1}^n v^{(0)}_{z^\alpha z^\beta}(y,0) \, \bar{z}^\beta(y) = -v^{(1)}_{z^\alpha}(y,0) 
\end{equation}
Clearly, $\bar{z}$ is a smooth function defined on $\Omega$. For each point $y \in \Omega$, we define 
\[V^{(2)}(y) = v^{(2)}(y,0) + \frac{1}{2} \sum_{\alpha=1}^n v^{(1)}_{z^\alpha}(y,0) \, \bar{z}^\alpha(y)\] 
and 
\begin{align*} 
V^{(3)}(y) 
&= v^{(3)}(y,0) + \sum_{\alpha=1}^n v^{(2)}_{z^\alpha}(y,0) \, \bar{z}^\alpha(y) + \frac{1}{2} \sum_{\alpha,\beta=1}^n v^{(1)}_{z^\alpha z^\beta}(y,0) \, \bar{z}^\alpha(y) \, \bar{z}^\beta(y) \\ 
&+ \frac{1}{6} \sum_{\alpha,\beta,\gamma=1}^n v^{(0)}_{z^\alpha z^\beta z^\gamma}(y,0) \, \bar{z}^\alpha(y) \, \bar{z}^\beta(y) \, \bar{z}^\gamma(y). 
\end{align*}
Note that $V^{(2)}$ and $V^{(3)}$ are smooth functions defined on $\Omega$.

\begin{lemma}
\label{relation.between.bar.w.and.bar.z}
The function $\bar{z}$ is related to the function $\bar{w}$ by the formula 
\[\sum_{\beta=1}^n \psi^\alpha_{z^\beta}(y,0) \, \bar{z}^\beta(y) = \bar{w}^\alpha(y)\] 
for each point $y \in \Omega$.
\end{lemma}

\textbf{Proof.} 
For each point $y \in \Omega$, we compute 
\begin{align*} 
&\sum_{\beta,\gamma,\delta=1}^n \Upsilon(y,0) \, u^{(0)}_{w^\gamma w^\delta}(y,0) \, \psi^\gamma_{z^\alpha}(y,0) \, \psi^\delta_{z^\beta}(y,0) \, \bar{z}^\beta(y) \\ 
&= \sum_{\beta=1}^n v^{(0)}_{z^\alpha z^\beta}(y,0) \, \bar{z}^\beta(y) \\ 
&= -v^{(1)}_{z^\alpha}(y,0) \\ 
&= -\sum_{\gamma=1}^n \Upsilon(y,0) \, u^{(1)}_{w^\gamma}(y,0) \, \psi^\gamma_{z^\alpha}(y,0) \\ 
&= \sum_{\gamma,\delta=1}^n \Upsilon(y,0) \, u^{(0)}_{w^\gamma w^\delta}(y,0) \, \psi^\gamma_{z^\alpha}(y,0) \, \bar{w}^\delta(y). 
\end{align*} 
The first equality follows from Lemma \ref{v0}. The second equality follows from (\ref{definition.of.bar.z}). The third equality follows from Lemma \ref{v1}. The fourth equality follows from (\ref{definition.of.bar.w}). 

For each point $y \in \Omega$, the matrix $(\psi^\gamma_{z^\alpha}(y,0))_{1 \leq \alpha,\gamma \leq n}$ and the matrix $(u^{(0)}_{w^\gamma w^\delta}(y,0))_{1 \leq \gamma,\delta \leq n}$ are invertible. Thus, we conclude that 
\[\sum_{\beta=1}^n \psi^\delta_{z^\beta}(y,0) \, \bar{z}^\beta(y) = \bar{w}^\delta(y)\] 
for each point $y \in \Omega$. This completes the proof of Lemma \ref{relation.between.bar.w.and.bar.z}. \\

\begin{lemma}
\label{relation.between.U2.and.V2}
We have $V^{(2)}(y) = \Upsilon(y,0) \, U^{(2)}(y)$ for each point $y \in \Omega$. 
\end{lemma}

\textbf{Proof.} 
Using Lemma \ref{v1} and Lemma \ref{relation.between.bar.w.and.bar.z}, we obtain 
\[\sum_{\alpha=1}^n v^{(1)}_{z^\alpha}(y,0) \, \bar{z}^\alpha(y) = \sum_{\beta=1}^n \Upsilon(y,0) \, u^{(1)}_{w^\beta}(y,0) \, \bar{w}^\beta(y)\] 
for each point $y \in \Omega$. Using the definitions of $U^{(2)}$ and $V^{(2)}$, we conclude that $V^{(2)}(y) = \Upsilon(y,0) \, U^{(2)}(y)$ for each point $y \in \Omega$. This completes the proof of Lemma \ref{relation.between.U2.and.V2}. \\

\begin{lemma}
\label{relation.between.U3.and.V3}
We have 
\begin{align*} 
V^{(3)}(y) 
&= \Upsilon(y,0) \, U^{(3)}(y) + \sum_{i=1}^m \sum_{\gamma=1}^n \Upsilon(y,0) \, U^{(2)}_{x^i}(y) \, \varphi^i_{z^\gamma}(y,0) \, \bar{z}^\gamma(y) \\ 
&+ \sum_{\gamma=1}^n \Upsilon_{z^\gamma}(y,0) \, U^{(2)}(y)  \, \bar{z}^\gamma(y) 
\end{align*}
for each point $y \in \Omega$. 
\end{lemma}

\textbf{Proof.} 
Using Lemma \ref{v0} and Lemma \ref{relation.between.bar.w.and.bar.z}, we obtain 
\begin{align*}
&\sum_{\alpha,\beta,\gamma=1}^n v^{(0)}_{z^\alpha z^\beta z^\gamma}(y,0) \, \bar{z}^\alpha(y) \, \bar{z}^\beta(y) \, \bar{z}^\gamma(y) \\ 
&= \sum_{\delta,\zeta,\eta=1}^n \Upsilon(y,0) \, u^{(0)}_{w^\delta w^\zeta w^\eta}(y,0) \, \bar{w}^\delta(y) \, \bar{w}^\zeta(y) \, \bar{w}^\eta(y) \\ 
&+ 3 \sum_{\alpha,\beta,\delta,\zeta=1}^n \Upsilon(y,0) \, u^{(0)}_{w^\delta w^\zeta}(y,0) \, \psi^\zeta_{z^\alpha z^\beta}(y,0) \, \bar{z}^\alpha(y) \, \bar{z}^\beta(y) \, \bar{w}^\delta(y) \\ 
&+ 3 \sum_{i=1}^m \sum_{\gamma,\delta,\zeta=1}^n \Upsilon(y,0) \, u^{(0)}_{x^i w^\delta w^\zeta}(y,0) \, \varphi^i_{z^\gamma}(y,0) \, \bar{w}^\delta(y) \, \bar{w}^\zeta(y) \, \bar{z}^\gamma(y) \\ 
&+ 3 \sum_{\gamma,\delta,\zeta=1}^n \Upsilon_{z^\gamma}(y,0) \, u^{(0)}_{w^\delta w^\zeta}(y,0) \, \bar{w}^\delta(y) \, \bar{w}^\zeta(y) \, \bar{z}^\gamma(y) 
\end{align*} 
for each point $y \in \Omega$. Moreover, Lemma \ref{v1} and Lemma \ref{relation.between.bar.w.and.bar.z} imply that  
\begin{align*} 
&\sum_{\alpha,\beta=1}^n v^{(1)}_{z^\alpha z^\beta}(y,0) \, \bar{z}^\alpha(y) \, \bar{z}^\beta(y) \\ 
&= \sum_{\gamma,\delta=1}^n \Upsilon(y,0) \, u^{(1)}_{w^\gamma w^\delta}(y,0) \, \bar{w}^\gamma(y) \, \bar{w}^\delta(y) \\ 
&+ \sum_{\alpha,\beta,\gamma=1}^n \Upsilon(y,0) \, u^{(1)}_{w^\gamma}(y,0) \, \psi^\gamma_{z^\alpha z^\beta}(y,0) \, \bar{z}^\alpha(y) \, \bar{z}^\beta(y) \\
&+ 2 \sum_{i=1}^m \sum_{\alpha,\gamma=1}^n \Upsilon(y,0) \, u^{(1)}_{x^i w^\gamma}(y,0) \, \varphi^i_{z^\alpha}(y,0) \, \bar{z}^\alpha(y) \, \bar{w}^\gamma(y) \\ 
&+ 2 \sum_{\alpha,\gamma=1}^n \Upsilon_{z^\alpha}(y,0) \, u^{(1)}_{w^\gamma}(y,0) \, \bar{z}^\alpha(y) \, \bar{w}^\gamma(y) 
\end{align*}
for each point $y \in \Omega$. Finally, using Lemma \ref{v2} and Lemma \ref{relation.between.bar.w.and.bar.z}, we find 
\begin{align*} 
&\sum_{\alpha=1}^n v^{(2)}_{z^\alpha}(y,0) \, \bar{z}^\alpha(y) \\ 
&= \sum_{\beta=1}^n \Upsilon(y,0) \, u^{(2)}_{w^\beta}(y,0) \, \bar{w}^\beta(y) \\ 
&+ \sum_{i=1}^m \sum_{\alpha=1}^n \Upsilon(y,0) \, u^{(2)}_{x^i}(y,0) \, \varphi^i_{z^\alpha}(y,0) \, \bar{z}^\alpha(y) \\ 
&+ \sum_{\alpha=1}^n \Upsilon_{z^\alpha}(y,0) \, u^{(2)}(y,0)  \, \bar{z}^\alpha(y) 
\end{align*}
for each point $y \in \Omega$. Using the definitions of $U^{(3)}$ and $V^{(3)}$ together with the identity (\ref{definition.of.bar.w}), we conclude that 
\begin{align*} 
V^{(3)}(y) 
&= \Upsilon(y,0) \, U^{(3)}(y) + \sum_{i=1}^m \sum_{\gamma=1}^n \Upsilon(y,0) \, u^{(2)}_{x^i}(y,0) \, \varphi^i_{z^\gamma}(y,0) \, \bar{z}^\gamma(y) \\ 
&+ \sum_{i=1}^m \sum_{\alpha,\gamma=1}^n \Upsilon(y,0) \, u^{(1)}_{x^i w^\alpha}(y,0) \, \varphi^i_{z^\gamma}(y,0) \, \bar{w}^\alpha(y) \, \bar{z}^\gamma(y) \\ 
&+ \frac{1}{2} \sum_{i=1}^m \sum_{\alpha,\beta,\gamma=1}^n \Upsilon(y,0) \, u^{(0)}_{x^i w^\alpha w^\beta}(y,0) \, \varphi^i_{z^\gamma}(y,0) \, \bar{w}^\alpha(y) \, \bar{w}^\beta(y) \, \bar{z}^\gamma(y) \\ 
&+ \sum_{\gamma=1}^n \Upsilon_{z^\gamma}(y,0) \, u^{(2)}(y,0)  \, \bar{z}^\gamma(y) \\ 
&+ \frac{1}{2} \sum_{\alpha,\gamma=1}^n \Upsilon_{z^\gamma}(y,0) \, u^{(1)}_{w^\alpha}(y,0) \, \bar{w}^\alpha(y) \, \bar{z}^\gamma(y)
\end{align*} 
for each point $y \in \Omega$. Using Lemma \ref{derivative.of.U2}, it follows that 
\begin{align*} 
V^{(3)}(y) 
&= \Upsilon(y,0) \, U^{(3)}(y) + \sum_{i=1}^m \sum_{\gamma=1}^n \Upsilon(y,0) \, U^{(2)}_{x^i}(y) \, \varphi^i_{z^\gamma}(y,0) \, \bar{z}^\gamma(y) \\ 
&+ \sum_{\gamma=1}^n \Upsilon_{z^\gamma}(y,0) \, U^{(2)}(y)  \, \bar{z}^\gamma(y) 
\end{align*}
for each point $y \in \Omega$. This completes the proof of Lemma \ref{relation.between.U3.and.V3}. \\

\begin{proposition} 
\label{invariance}
Let $\Lambda$ be a subset of $\Omega$ with the property that $V^{(2)} = 0$ and $dV^{(2)} = 0$ at each point in $\Lambda$. Then $V^{(3)}(y) = \Upsilon(y,0) \, U^{(3)}(y)$ for each point $y \in \Lambda$. In particular, the function $\Upsilon(\cdot,0)^{-1} \, V^{(3)}|_\Lambda$ is the restriction of a smooth function on $\bar{B}_{r_0}^m$.
\end{proposition}

\textbf{Proof.} 
It follows from Lemma \ref{relation.between.U2.and.V2} that $U^{(2)} = 0$ and $dU^{(2)} = 0$ at each point $y \in \Lambda$. The assertion follows now from Lemma \ref{relation.between.U3.and.V3}. \\

\begin{corollary} 
\label{positivity}
Let $\rho: \bar{B}_{r_0}^m \to \mathbb{R}$ be a nonnegative continuous function. We assume that $\kappa$ is a positive real number with the following properties: 
\begin{itemize} 
\item For each point $y \in \Omega$, we have $V^{(2)}(y) \geq \kappa \, \Upsilon(y,0) \, \rho(y)$. 
\item For each point $y \in \Omega \cap \{\rho = 0\}$, we have $V^{(2)}(y) = 0$ and $V^{(3)}(y) \geq \kappa \, \Upsilon(y,0)$.
\end{itemize} 
Let $A \subset \bar{B}_{r_0}^m$ be a compact set with the property that $A$ is contained in the closure of $\Omega$ and $A \cap \{\rho = 0\}$ is contained in the closure of $\Omega \cap \{\rho = 0\}$. Then $\inf_{y \in A} U(y,s) > 0$ if $s \in (0,s_0]$ is sufficiently small.
\end{corollary}

\textbf{Proof.} 
Using Lemma \ref{relation.between.U2.and.V2}, we obtain $V^{(2)}(y) = \Upsilon(y,0) \, U^{(2)}(y)$ for each point $y \in \Omega$. Since $V^{(2)}(y) \geq \kappa \, \Upsilon(y,0) \, \rho(y)$ for each point $y \in \Omega$, it follows that $U^{(2)}(y) \geq \kappa \, \rho(y)$ for each point $y \in \Omega$. Since $U^{(2)}$ and $\rho$ are continuous functions on $\bar{B}_{r_0}^m$ and $A$ is contained in the closure of $\Omega$, we conclude that $U^{(2)}(y) \geq \kappa \, \rho(y)$ for each point $y \in A$.

By assumption, $V^{(2)}$ is a nonnegative smooth function on $\Omega$, and $V^{(2)} = 0$ at each point in $\Omega \cap \{\rho = 0\}$. Consequently, $dV^{(2)} = 0$ at each point in $\Omega \cap \{\rho = 0\}$. Using Proposition \ref{invariance}, we obtain $V^{(3)}(y) = \Upsilon(y,0) \, U^{(3)}(y)$ for each point $y \in \Omega \cap \{\rho = 0\}$. Since $V^{(3)}(y) \geq \kappa \, \Upsilon(y,0)$ for each point $y \in \Omega \cap \{\rho = 0\}$, it follows that $U^{(3)}(y) \geq \kappa$ for each point $y \in \Omega \cap \{\rho = 0\}$. Since $U^{(3)}$ is a continuous function on $\bar{B}_{r_0}^m$ and $A \cap \{\rho = 0\}$ is contained in the closure of $\Omega \cap \{\rho = 0\}$, we conclude that $U^{(3)}(y) \geq \kappa$ for each point $y \in A \cap \{\rho=0\}$. 

By continuity, we can find a positive real number $\varepsilon$ such that $U^{(3)}(y) \geq \frac{1}{2} \, \kappa$ for all $y \in A \cap \{0 \leq \rho \leq \varepsilon\}$. Using Lemma \ref{asymptotic.expansion.of.U}, we obtain 
\begin{align*} 
U(y,s) 
&\geq s^2 \, U^{(2)}(y) + s^3 \, U^{(3)}(y) - C \, s^4 \\ 
&\geq s^3 \, U^{(3)}(y) - C \, s^4 \\ 
&\geq \frac{1}{2} \, s^3 \, \kappa - C \, s^4 
\end{align*}
for all $y \in A \cap \{0 \leq \rho \leq \varepsilon\}$ and all $s \in (0,s_0]$. Moreover, Lemma \ref{asymptotic.expansion.of.U} gives 
\begin{align*} 
U(y,s) 
&\geq s^2 \, U^{(2)}(y) - C \, s^3 \\ 
&\geq s^2 \, \kappa \, \rho(y) - C \, s^3 \\ 
&\geq s^2 \, \kappa \, \varepsilon - C \, s^3 
\end{align*}
for all $y \in A \setminus \{0 \leq \rho \leq \varepsilon\}$ and all $s \in (0,s_0]$. Thus, 
\[\inf_{y \in A} U(y,s) \geq \min \Big \{ \frac{1}{2} \, s^3 \, \kappa - C \, s^4,s^2 \, \kappa \, \varepsilon - C \, s^3 \Big \}\] 
for all $s \in (0,s_0]$. This completes the proof of Corollary \ref{positivity}. \\

In the remainder of this section, we state two results that apply in the manifold setting. 

\begin{theorem} 
\label{smooth.extension}
Let $\mathcal{N}$ be a compact Riemannian manifold of dimension $m+n$, and let $\mathcal{Z}$ be a compact submanifold of $\mathcal{N}$ of dimension $m$. Let $s_0$ be a positive real number, and let $u: \mathcal{N} \times [-s_0,s_0] \to \mathbb{R}$ be a smooth function. We assume that 
\[u(\cdot,s) = u^{(0)} + s \, u^{(1)} + s^2 \, u^{(2)} + s^3 \, u^{(3)} + O(s^4),\]
where $u^{(0)},u^{(1)},u^{(2)},u^{(3)}: \mathcal{N} \to \mathbb{R}$ are smooth functions. We assume that $u^{(0)}|_{\mathcal{Z}} = 0$ and $u^{(0)} \geq \sigma \, d(\cdot,\mathcal{Z})^2$ at each point in $\mathcal{N}$, where $\sigma$ is a positive constant. Moreover, we assume that $u^{(1)}|_{\mathcal{Z}} = 0$.

Let $\Omega$ be an open subset of $\mathcal{Z}$ and let $r$ be a positive real number. Suppose that $\Upsilon: \Omega \times B_r^n \to \mathbb{R}$ is a positive smooth function. Moreover, suppose that $\Phi: \Omega \times B_r^n \to \mathcal{N}$ is a smooth map with the property that $\Phi(y,0) = y$ for each point $y \in \Omega$, and the differential $D\Phi_{(y,0)}: T_y \mathcal{Z} \times \mathbb{R}^n \to T_y \mathcal{N}$ is invertible for each point $y \in \Omega$. We define smooth functions $v^{(0)},v^{(1)},v^{(2)},v^{(3)}: \Omega \times B_r^n \to \mathbb{R}$ by 
\[v^{(0)} = \Upsilon \, (u^{(0)} \circ \Phi),\ v^{(1)} = \Upsilon \, (u^{(1)} \circ \Phi),\ v^{(2)} = \Upsilon \, (u^{(2)} \circ \Phi),\ v^{(3)} = \Upsilon \, (u^{(3)} \circ \Phi).\] 
Let $V^{(2)}: \Omega \to \mathbb{R}$ and $V^{(3)}: \Omega \to \mathbb{R}$ be defined as above. Let $\Lambda$ be a subset of $\Omega$ with the property that $V^{(2)} = 0$ and $dV^{(2)} = 0$ at each point in $\Lambda$. Then the function $\Upsilon(\cdot,0)^{-1} \, V^{(3)}|_\Lambda$ is the restriction of a smooth function on $\mathcal{Z}$.
\end{theorem}

\textbf{Proof.} We cover $\mathcal{Z}$ by finitely many small balls and apply Proposition \ref{invariance} on each ball. \\

\begin{theorem} 
\label{positivity.manifold.version}
Let $\mathcal{N}$ be a compact Riemannian manifold of dimension $m+n$, and let $\mathcal{Z}$ be a compact submanifold of $\mathcal{N}$ of dimension $m$. Let $s_0$ be a positive real number, and let $u: \mathcal{N} \times [-s_0,s_0] \to \mathbb{R}$ be a smooth function. We assume that 
\[u(\cdot,s) = u^{(0)} + s \, u^{(1)} + s^2 \, u^{(2)} + s^3 \, u^{(3)} + O(s^4),\]
where $u^{(0)},u^{(1)},u^{(2)},u^{(3)}: \mathcal{N} \to \mathbb{R}$ are smooth functions. We assume that $u^{(0)}|_{\mathcal{Z}} = 0$ and $u^{(0)} \geq \sigma \, d(\cdot,\mathcal{Z})^2$ at each point in $\mathcal{N}$, where $\sigma$ is a positive constant. Moreover, we assume that $u^{(1)}|_{\mathcal{Z}} = 0$.

Let $\Omega$ be an open subset of $\mathcal{Z}$ and let $r$ be a positive real number. Suppose that $\Upsilon: \Omega \times B_r^n \to \mathbb{R}$ is a positive smooth function. Moreover, suppose that $\Phi: \Omega \times B_r^n \to \mathcal{N}$ is a smooth map with the property that $\Phi(y,0) = y$ for each point $y \in \Omega$, and the differential $D\Phi_{(y,0)}: T_y \mathcal{Z} \times \mathbb{R}^n \to T_y \mathcal{N}$ is invertible for each point $y \in \Omega$. We define smooth functions $v^{(0)},v^{(1)},v^{(2)},v^{(3)}: \Omega \times B_r^n \to \mathbb{R}$ by 
\[v^{(0)} = \Upsilon \, (u^{(0)} \circ \Phi),\ v^{(1)} = \Upsilon \, (u^{(1)} \circ \Phi),\ v^{(2)} = \Upsilon \, (u^{(2)} \circ \Phi),\ v^{(3)} = \Upsilon \, (u^{(3)} \circ \Phi).\] 
Let $V^{(2)}: \Omega \to \mathbb{R}$ and $V^{(3)}: \Omega \to \mathbb{R}$ be defined as above. Let $\rho: \mathcal{Z} \to \mathbb{R}$ be a nonnegative continuous function. We assume that $\kappa$ is a positive real number with the following properties:
\begin{itemize} 
\item For each point $y \in \Omega$, we have $V^{(2)}(y) \geq \kappa \, \Upsilon(y,0) \, \rho(y)$. 
\item For each point $y \in \Omega \cap \{\rho = 0\}$, we have $V^{(2)}(y) = 0$ and $V^{(3)}(y) \geq \kappa \, \Upsilon(y,0)$.
\end{itemize} 
Finally, we assume that $\Omega$ is dense in $\mathcal{Z}$ and $\Omega \cap \{\rho = 0\}$ is dense in $\{\rho = 0\}$. Then $\inf_{\mathcal{N}} u(\cdot,s) > 0$ if $s \in (0,s_0]$ is sufficiently small.
\end{theorem}

\textbf{Proof.} We cover $\mathcal{Z}$ by finitely many small balls and apply Corollary \ref{positivity} on each ball. \\

\section{The Cheeger-M\"uter metric}\label{sec:Mueter.metric}

Let $M = S^2 \times S^2$. This section computes the Cheeger-M\"uter metric on $M$, which serves as the background metric of our construction.

For a vector $\omega\in\mathbb{R}^3$, we denote by $K_\omega$ and $J_\omega$ the vector fields on $M$ given at the point $p=(p_1,p_2)$ by
\[
K_\omega = (\omega\wedge p_1, \omega\wedge p_2),\ J_\omega=(\omega\wedge p_1, -\omega\wedge p_2).
\]
More generally, if $\tilde{\omega}$ is an $\mathbb{R}^3$-valued function defined on an open subset $U$ of $M$, let $K_{\tilde{\omega}}$ denote the vector field on $U$ given at the point $p=(p_1,p_2)$ by
\[
K_{\tilde{\omega}}= (\tilde{\omega}(p)\wedge p_1,\tilde{\omega}(p)\wedge p_2).
\]
As above, we define $M_{\text{\rm nondeg}} = M \setminus (\Delta_+ \cup \Delta_-)$. We next explain how $M_{\text{\rm nondeg}}$ can be identified with $SO(3)\times (0,\frac{\pi}{2})$. It is convenient first to fix some notation for $SO(3)$.
We view $SO(3)$ as the group of orthogonal $3\times 3$ matrices with determinant $1$.
Each element $\mathbf{q}$ of $SO(3)$ can be written as $\mathbf{q}=(q_1,q_2,q_3)$, where $q_1,q_2,q_3$ form an orthonormal basis of $\mathbb{R}^3$ that satisfies $q_1\wedge q_2=q_3$.
For $i=1,2,3$, let $E_i$ be the left-invariant vector field on $SO(3)$ characterized by
\[
E_iq_j=\sum_{k=1}^3\varepsilon_{ijk}q_k,
\]
where $\varepsilon_{ijk}$ is the totally antisymmetric symbol with $\varepsilon_{123}=1$.
Let $\sigma^i$, $i=1,2,3$, be the left-invariant one-forms dual to $E_1, E_2,$ and $E_3$.

With this notation in place, the identification of $M_{\text{\rm nondeg}}$ with $SO(3)\times (0,\frac{\pi}{2})$ is given by $(p_1,p_2)\mapsto (q_1,q_2,q_3,\theta)$, where
\begin{equation}\label{identify.s2xs2.SO(3)}
q_1=\frac{p_1+p_2}{|p_1+p_2|},\ q_2=\frac{p_2-p_1}{|p_2-p_1|},\ q_3=\frac{p_1\wedge p_2}{|p_1\wedge p_2|},\ \theta=\frac{1}{2}\arccos(\langle p_1,p_2\rangle).
\end{equation}
Under this identification, $E_i$ and $\sigma^i$ can be viewed as vector fields and one-forms on $M_{\text{\rm nondeg}}$.
Explicitly,
\[
E_i=(q_i\wedge p_1, q_i\wedge p_2)=K_{q_i}, i=1,2,3,
\]
and
\[
\partial_\theta=(-q_3\wedge p_1,q_3\wedge p_2)=-J_{q_3}.
\]
Using $p_1=\cos\theta q_1-\sin\theta q_2$ and $p_2=\cos\theta q_1+\sin\theta q_2$, these vector fields take the form
\begin{align*}
E_1=& (-\sin\theta q_3, \sin\theta q_3),\\
E_2=&(-\cos\theta q_3, -\cos\theta q_3),\\
E_3=& (\sin\theta q_1+\cos\theta q_2, -\sin\theta q_1+\cos\theta q_2),\\
\partial_\theta=& (-\sin\theta q_1-\cos\theta q_2,-\sin\theta q_1+\cos\theta q_2).
\end{align*}
Let $g_{\text{\rm std}}$ be the standard product metric on $M$, with each factor having Gauss curvature one.
The above expressions give
\[
g_{\text{\rm std}}=2\sin^2\theta \, \sigma^1\otimes \sigma^1+2\cos^2\theta \, \sigma^2\otimes \sigma^2+2 \, \sigma^3\otimes \sigma^3+2 \, d\theta\otimes d\theta.\]
The Cheeger-M\"uter metric can now be computed in the coordinates given by \eqref{identify.s2xs2.SO(3)}.
Consider the diagonal action of $SO(3)$ on $M = S^2\times S^2$.
The Lie algebra $\mathfrak{so}(3)$ is identified with $\mathbb{R}^3$ so that, for $\omega\in \mathfrak{so}(3)\cong\mathbb{R}^3$,
the associated vector field with the diagonal action is given by $K_\omega$.
Let $Q$ be the bi-invariant metric on $SO(3)$ that corresponds to the standard inner product on $\mathbb{R}^3$ under the identification $\mathfrak{so}(3)\cong \mathbb{R}^3$.
Fix $t>0$ and equip $M \times SO(3)$ with the metric $g_{\text{\rm std}}+\frac{1}{t}Q$.
The Cheeger deformation of $g_{\text{\rm std}}$ with parameter $t$, denoted by $g_{\text{\rm CM},t}$, is characterized by the requirement that the map from $(M \times SO(3),g_{\text{\rm std}}+\frac{1}{t}Q)$ to $(M,g_{\text{\rm CM},t})$ given by $(p,\mathbf{q})\mapsto \mathbf{q}^{-1}\cdot p$ is a Riemannian submersion.
For $p\in M$, let $\mathcal{V}_p$ be the vertical subspace of $T_p M$ spanned by $K_\omega,\ \omega\in\mathbb{R}^3$,
 and let $\mathcal{H}_p$ denote the $g_{\text{\rm std}}$-orthogonal complement of $\mathcal{V}_p$.
Let $P_p:\mathbb{R}^3\to \mathbb{R}^3$ be the map such that
\[\langle P_p(\omega_1), \omega_2\rangle= g_{\text{\rm std}}(K_{\omega_1},K_{\omega_2})\ \text{\rm at}\ p.\]
This map $P_p$ can be viewed as an endomorphism on $\mathcal{V}_p$ through $P_p(K_\omega)=K_{P_p(\omega)}$.
Define the endomorphism $C_t$ of $T_p M$ by 
\[C_tX= (I+tP_p)^{-1}X^{\mathcal{V}}+X^{\mathcal{H}},\]
where $X=X^{\mathcal{V}}+X^{\mathcal{H}}$ is the vertical-horizontal decomposition.
The formula for the Cheeger deformation \cite{Ziller} gives
\[g_{\text{\rm CM},t}(X,Y)= g_{\text{\rm std}}(C_t X, Y)\ \text{\rm for all}\ X,Y\in T_p M.\]
Since $E_i=K_{q_i}$, the vectors $q_1,q_2,q_3\in\mathbb{R}^3$, and equivalently $E_1,E_2,E_3\in T_p M$ under the induced endomorphism of $\mathcal{V}_p$, are eigenvectors of $P_p$ with eigenvalues $2\sin^2\theta,\ 2\cos^2\theta, 2$, respectively.
It follows that
\[g_{\text{\rm CM},t}=\frac{2\sin^2\theta}{1+2t\sin^2\theta} \, \sigma^1\otimes \sigma^1+\frac{2\cos^2\theta}{1+2t\cos^2\theta} \, \sigma^2\otimes \sigma^2
+\frac{2}{1+2t} \, \sigma^3\otimes \sigma^3 + 2 \, d\theta\otimes d\theta.\]
Throughout the rest of the paper, $g=\frac{1}{2}g_{\text{\rm CM},1}$ is the background metric on $M$.
Explicitly, on $M_{\text{\rm nondeg}}$,
\[
g=\frac{\sin^2\theta}{2-\cos(2\theta)}\sigma^1\otimes \sigma^1
+\frac{\cos^2\theta}{2+\cos(2\theta)}\sigma^2\otimes \sigma^2
+\frac{1}{3}\sigma^3\otimes \sigma^3+d\theta\otimes d\theta.
\]
We define vector fields $m_i$, $1 \leq i \leq 4$, on $M_{\text{\rm nondeg}}$ by
\[
m_1=\frac{1}{\sin\theta} E_1,\ m_2=\frac{1}{\cos\theta} E_2,\ m_3=E_3,\ m_4=\partial_\theta.
\] 
Then 
\[|m_1|_g^2 = \frac{1}{2-\cos(2\theta)},\ |m_2|_g^2 = \frac{1}{2+\cos(2\theta)},\ |m_3|_g^2 = \frac{1}{3},\ |m_4|_g^2 = 1\]
on $M_{\text{\rm nondeg}}$. Writing $\omega=\sum_{i=1}^3 \langle q_i,\omega \rangle q_i$ and using $p_1=\cos\theta q_1-\sin\theta q_2$ and $p_2=\cos\theta q_1+\sin\theta q_2$, we have
\begin{align*}
\omega\wedge p_1=& \langle q_1,\omega \rangle (-\sin\theta q_3)+\langle q_2,\omega \rangle (-\cos\theta q_3)+\langle q_3,\omega \rangle (\sin\theta q_1+\cos\theta q_2),\\
\omega\wedge p_2=& \langle q_1,\omega \rangle (\sin\theta q_3)+\langle q_2,\omega \rangle(-\cos\theta q_3)+\langle q_3,\omega \rangle(-\sin\theta q_1+\cos\theta q_2).
\end{align*}
It follows that
\begin{equation}\label{K.omega.in.mi}
K_\omega=\langle q_1,\omega \rangle\sin\theta m_1+\langle q_2,\omega \rangle\cos\theta m_2+\langle q_3,\omega \rangle m_3,    
\end{equation}
and
\begin{equation}\label{J.omega.in.mi}
J_\omega= \langle q_2,\omega \rangle\cos\theta m_1+\langle q_1,\omega \rangle\sin\theta m_2-\langle q_3,\omega \rangle m_4. 
\end{equation}

In the remainder of this section, we introduce some notation that we will need later. We denote by $\mathcal{N}$ the set of all two-dimensional tangent planes. Note that $\mathcal{N}$ is a compact smooth manifold of dimension $8$. Given a point $p = (p_1,p_2) \in M$, we denote by $\mathcal{Z}_p$ the set of all two-planes $\pi \subset T_p M$ with the property that there exists a unit vector $b \in S^2$ such that $\langle p_1,b \rangle = \langle p_2,b \rangle = 0$ and $\pi = \text{\rm span}\{(b,0),(0,b)\}^\perp$, where the orthogonal complement is taken with respect to $g_{\text{\rm std}}$. If we define $\mathcal{Z} = \bigcup_{p \in M} \mathcal{Z}_p$, then $\mathcal{Z}$ is a compact submanifold of $\mathcal{N}$ of dimension $4$. Note that $\mathcal{Z}$ is the set of all two-planes that have zero sectional curvature with respect to the metric $g$.

Let $\mathcal{Z}_{\text{\rm nondeg}} = \bigcup_{p \in M_{\text{\rm nondeg}}} \mathcal{Z}_p$. It is easy to see that $\mathcal{Z}_{\text{\rm nondeg}}$ is an open and dense subset of $\mathcal{Z}$. \\

\section{Formulae for metric perturbations}

\label{sec:formula.metric.perturbation}

We consider a metric perturbation
\[g_s = g + s \, h^{(1)} + s^2 \, h^{(2)}.\] 
Let $R_{g_s}(X,Y,Z,W)$ be the Riemannian curvature tensor for $g_s$.
We denote by $R^{(0)},R^{(1)},R^{(2)},R^{(3)}$ the coefficients in the Taylor expansion of $R_{g_s}$:
\begin{align*}
R_{g_s}(X,Y,Z,W)=&R^{(0)}(X,Y,Z,W)+sR^{(1)}(X,Y,Z,W)\\
&+s^2R^{(2)}(X,Y,Z,W)+s^3R^{(3)}(X,Y,Z,W)+O(s^4). 
\end{align*}
Given a symmetric $(0,2)$-tensor $h$, we define a $(0,3)$-tensor $Th$ and a $(0,4)$-tensor $Lh$ by
\[(Th)_{ijk}=\frac{1}{2}\big( D_i h_{jk}+D_j h_{ki}-D_k h_{ij} \big),\]
and
\begin{align*} 
(Lh)_{ijkl}=&\frac{1}{2}\big(- (D^2_{i,k}h)_{jl}+(D^2_{i,l} h)_{jk}+(D^2_{j,k}h)_{il}-(D^2_{j,l}h)_{ik} \big)\\
&+\frac{1}{2}g^{pq}R_{ijkp}h_{ql}-\frac{1}{2}g^{pq}R_{ijlp}h_{kq}. 
\end{align*}
Given two symmetric $(0,2)$-tensors $h_{\mathrm{I}}$ and $h_{\mathrm{II}}$, we define a $(0,4)$-tensor $Q(h_{\mathrm{I}},h_{\mathrm{II}})$ by
\begin{align*}
Q(h_{\mathrm{I}},h_{\mathrm{II}})_{ijkl}=&\frac{1}{2}g^{pq}(Th_{\mathrm{I}})_{jkp}(Th_{\mathrm{II}})_{ilq}-\frac{1}{2}g^{pq}(Th_{\mathrm{I}})_{ikp}(Th_{\mathrm{II}})_{jlq}\\
+&\frac{1}{2}g^{pq}(Th_{\mathrm{II}})_{jkp}(Th_{\mathrm{I}})_{ilq}-\frac{1}{2}g^{pq}(Th_{\mathrm{II}})_{ikp}(Th_{\mathrm{I}})_{jlq}. 
\end{align*}
Note that $Q(h_{\mathrm{I}},h_{\mathrm{II}}) = Q(h_{\mathrm{II}},h_{\mathrm{I}})$. Given three symmetric $(0,2)$-tensors $h_{\mathrm{I}}$, $h_{\mathrm{II}}$ and $h_{\mathrm{III}}$, we define a $(0,4)$-tensor $C(h_{\mathrm{I}},h_{\mathrm{II}},h_{\mathrm{III}})$ by
\begin{align*}
C(h_{\mathrm{I}},h_{\mathrm{II}},h_{\mathrm{III}})_{ijkl}=&
-g^{pp'}g^{qq'}(h_{\mathrm{I}})_{p'q'}(Th_{\mathrm{II}})_{jkp}(Th_{\mathrm{III}})_{ilq}\\
&+g^{pp'}g^{qq'}(h_{\mathrm{I}})_{p'q'}(Th_{\mathrm{II}})_{ikp}(Th_{\mathrm{III}})_{jlq}.
\end{align*}
From Proposition~\ref{curvature.tensor.of.perturbed.metric}, we have the formulas for $R^{(1)}$, $R^{(2)}$, and $R^{(3)}$ below. 
\begin{lemma}
\label{curvature.tensor.expansion}
We have 
\begin{align*}
R^{(1)}=&Lh^{(1)},\\
R^{(2)}=&Lh^{(2)}+Q(h^{(1)},h^{(1)}),\\
R^{(3)}=&2Q(h^{(1)},h^{(2)})+C(h^{(1)},h^{(1)},h^{(1)}).
\end{align*}
\end{lemma}
Let $m_i$, $1 \leq i\leq 4$, denote the vector fields on $M_{\text{\rm nondeg}}$ defined in Section~\ref{sec:Mueter.metric}.

Given a point $y \in M_{\text{\rm nondeg}}$ and a vector $z \in \mathbb{R}^4$, we consider the two-plane $\text{\rm span}\{m_3(z),m_4(z)\} \subset T_y M$, where 
\[ m_3(z)=m_3+z^1 m_1+z^3 m_2,\ m_4(z)=m_4+z^4 m_1+z^2 m_2. \] 
In the following, we assume that $s_0$ is sufficiently small so that $g_s$ is a Riemannian metric on $M$ for $|s|\leq s_0$. We define a function $v: M_{\text{\rm nondeg}}\times\mathbb{R}^4\times [-s_0,s_0]$ through
\[v(y,z,s)= R_{g_s}(m_3(z),m_4(z),m_3(z),m_4(z))\ \text{\rm at}\ y.\]
Let $v^{(k)}(y,z)$, $0\leq k\leq 3$, be the coefficients in the Taylor expansion of $v(y,z,s)$ with respect to $s$ as in Section~\ref{framework}.
Let $V^{(2)}(y)$ and $V^{(3)}(y)$ be defined as in Section~\ref{framework}.

For each fixed $y \in M_{\text{\rm nondeg}}$, the function 
\[z \mapsto v^{(k)}(y,z)=R^{(k)}(m_3(z),m_4(z),m_3(z),m_4(z))\]
is a polynomial in $z\in\mathbb{R}^4$ of degree at most $4$. It is then straightforward to compute $v^{(k)}(y,0)$, $(v^{(k)}_{z^\alpha}(y,0))_{1\leq \alpha\leq 4}$, and $(v^{(k)}_{z^\alpha z^\beta}(y,0))_{1\leq \alpha,\beta\leq 4}$. 
It is convenient to define projection operators $P_0$, $P_1$, and $P_2$ as follows.
Given a $(0,4)$-tensor $S$ on $M$, we define $S_{ijkl}=S(m_i,m_j,m_k,m_l)$ at each point on $M_{\text{\rm nondeg}}$. We define the projection operators
$$P_0S=S_{3434},$$
$$ P_1S= 2 \, (S_{1434},S_{3234},S_{2434},S_{3134}), $$
and
\[P_2S = 2 \begin{bmatrix} S_{1414} & S_{1234}+S_{1432} & S_{1424} & S_{1431}\\
S_{1234}+S_{1432} & S_{3232} & S_{2432} & S_{3132}\\
S_{1424} & S_{2432} & S_{2424} & S_{2134}+S_{2431}\\
S_{1431} & S_{3132} & S_{2134}+S_{2431} & S_{3131} \end{bmatrix}.\] 
Note that $P_0 S$ is a scalar function on $M_{\text{\rm nondeg}}$, $P_1 S$ is a function on $M_{\text{\rm nondeg}}$ taking values in $\mathbb{R}^4$, and $P_2 S$ is a function on $M_{\text{\rm nondeg}}$ taking values in $\mathbb{R}^{4 \times 4}$. With this understood,
\[v^{(k)}(y,0)=P_0R^{(k)},\]
\[(v^{(k)}_{z^\alpha}(y,0))_{1\leq \alpha\leq 4}=P_1R^{(k)},\]
and
\[(v^{(k)}_{z^\alpha z^\beta}(y,0))_{1\leq \alpha,\beta\leq 4}=P_2R^{(k)}\]
for each $k$. In particular, in view of Lemma~\ref{curvature.tensor.expansion}, we have
\[v^{(1)}(y,0)=P_0Lh^{(1)},\] 
\[v^{(2)}(y,0)=P_0Lh^{(2)}+P_0Q(h^{(1)},h^{(1)}),\] 
\[v^{(3)}(y,0)=2P_0Q(h^{(1)},h^{(2)})+P_0C(h^{(1)},h^{(1)},h^{(1)}),\]
\[(v^{(1)}_{z^\alpha}(y,0))_{1\leq \alpha\leq 4}=P_1Lh^{(1)},\]
\[ (v^{(2)}_{z^\alpha}(y,0))_{1\leq \alpha\leq 4}=P_1Lh^{(2)}+P_1Q(h^{(1)},h^{(1)}), \]
\[ (v^{(0)}_{z^\alpha z^\beta}(y,0))_{1\leq \alpha,\beta\leq 4}=P_2R^{(0)}, \]
\[ (v^{(1)}_{z^\alpha z^\beta}(y,0))_{1\leq \alpha,\beta\leq 4}=P_2Lh^{(1)}. \] 
For later convenience, we write $H = P_2R^{(0)}$. Given a symmetric $(0,2)$-tensor $h$, we write $r(h) = P_1Lh$, and we define 
\[z(h)=(z^1(h),z^2(h),z^3(h),z^4(h)) \]
by
\[ \sum_{\beta=1}^4 H_{\alpha\beta}z^\beta(h)=-r_\alpha(h).\]
Then
\[(v^{(1)}_{z^\alpha}(y,0))_{1\leq \alpha\leq 4}=r(h^{(1)}),\] 
\[ (v^{(2)}_{z^\alpha}(y,0))_{1\leq \alpha\leq 4}=r(h^{(2)})+P_1Q(h^{(1)},h^{(1)}),\]
and 
\[(\bar{z}^{\alpha}(y))_{1\leq \alpha\leq 4}=z(h^{(1)}).\]

\begin{lemma}\label{Vk.formula}
We have
\[ V^{(2)}=P_0Lh^{(2)}+P_0Q(h^{(1)},h^{(1)})+\frac{1}{2}r(h^{(1)}) \cdot z (h^{(1)})\] 
and 
\begin{align*}
V^{(3)}=&2P_0Q(h^{(1)},h^{(2)})+r(h^{(2)})\cdot z(h^{(1)})\\
&+P_0C(h^{(1)},h^{(1)},h^{(1)})+P_1Q(h^{(1)},h^{(1)})\cdot z(h^{(1)})\\
&+\frac{1}{2}  z(h^{(1)})^T\cdot P_2Lh^{(1)}\cdot z(h^{(1)})\\
&+\frac{1}{6}\sum_{\alpha,\beta,\gamma=1}^4 v^{(0)}_{z^\alpha z^\beta z^\gamma}(y,0) z^{\alpha}(h^{(1)})z^{\beta}(h^{(1)})z^{\gamma}(h^{(1)}).
\end{align*}
\end{lemma}

\textbf{Proof.}
The quantity $V^{(2)}(y)$ is defined by
\[V^{(2)}(y) = v^{(2)}(y,0) + \frac{1}{2} \sum_{\alpha=1}^4 v^{(1)}_{z^\alpha}(y,0) \, \bar{z}^\alpha(y).\] 
From the above discussion,
\[v^{(2)}(y,0)=P_0Lh^{(2)}+P_0Q(h^{(1)},h^{(1)})\]
and
\[\frac{1}{2} \sum_{\alpha=1}^4 v^{(1)}_{z^\alpha}(y,0) \, \bar{z}^\alpha(y)=\frac{1}{2} r(h^{(1)})\cdot z(h^{(1)}).\]
Putting these identities together gives the formula for $V^{(2)}$.

The quantity $V^{(3)}(y)$ is defined by
\begin{align*} 
V^{(3)}(y) 
&= v^{(3)}(y,0) + \sum_{\alpha=1}^4 v^{(2)}_{z^\alpha}(y,0) \, \bar{z}^\alpha(y) + \frac{1}{2} \sum_{\alpha,\beta=1}^4 v^{(1)}_{z^\alpha z^\beta}(y,0) \, \bar{z}^\alpha(y) \, \bar{z}^\beta(y) \\ 
&+ \frac{1}{6} \sum_{\alpha,\beta,\gamma=1}^4 v^{(0)}_{z^\alpha z^\beta z^\gamma}(y,0) \, \bar{z}^\alpha(y) \, \bar{z}^\beta(y) \, \bar{z}^\gamma(y). 
\end{align*}
From the above discussion,
\[v^{(3)}(y,0)=2P_0Q(h^{(1)},h^{(2)})+P_0C(h^{(1)},h^{(1)},h^{(1)}),\]
\[\sum_{\alpha=1}^4 v^{(2)}_{z^\alpha}(y,0) \, \bar{z}^\alpha(y)= r(h^{(2)})\cdot z(h^{(1)})+P_1Q(h^{(1)},h^{(1)})\cdot z(h^{(1)}),\]
and
\[\frac{1}{2} \sum_{\alpha,\beta=1}^4 v^{(1)}_{z^\alpha z^\beta}(y,0) \, \bar{z}^\alpha(y) \, \bar{z}^\beta(y)= \frac{1}{2}  z(h^{(1)})^T\cdot P_2Lh^{(1)}\cdot z(h^{(1)}).\]
Putting these together yields the formula for $V^{(3)}$. This completes the proof of Lemma \ref{Vk.formula}. \\

\section{A second order deformation of the Cheeger-M\"uter metric}

\label{second.order.deformation}

In this section, we define a particular second order deformation of the metric $g = \frac{1}{2} \, g_{\text{\rm CM},1}$, where $g_{\text{\rm CM},1}$ denotes the Cheeger-M\"uter metric. We work primarily in the coordinates given by \eqref{identify.s2xs2.SO(3)} and must address whether the objects defined in these coordinates extend smoothly to all of $M$.

Note that $\langle p_1,p_2\rangle=\cos(2\theta)$.
Since $\cos(2n\theta)$ can be expressed as a polynomial of degree $n$ in $\cos(2\theta)$ for every $n\in\mathbb{N}$, any linear combination of $\cos(2n\theta)$ extends uniquely to a globally smooth function on $M$.
We abuse notation and denote this extension by the same expression.

Recall that $K_\omega$, $J_\omega$, and, more generally, $K_{\tilde{\omega}}$ are vector fields defined at the beginning of Section~\ref{sec:Mueter.metric}.
Let $K^\flat_\omega$, $J^\flat_\omega$, and $K^\flat_{\tilde{\omega}}$ be their associated one-forms with respect to the metric $g$:
\[
K^\flat_\omega(X)=g(K_\omega,X),\ J^\flat_\omega(X)=g(J_\omega,X),\ K^\flat_{\tilde{\omega}}(X)=g(K_{\tilde{\omega}},X).
\]
Let $e_x, e_y, e_z$ be the standard basis of $\mathbb{R}^3$.
For two one-forms $\alpha,\beta$, write $\alpha\odot \beta$ for the symmetric tensor product
\[
\alpha\odot\beta=\frac{1}{2}(\alpha\otimes \beta+\beta\otimes \alpha).
\]
With these preparations, we are ready to define $h^{(1)}$ and $h^{(2)}$.

We define
\[h^{(1)}_a=\frac{6}{5} (2- \cos(2\theta) ) K^\flat_{p_1+p_2}\odot K^\flat_{e_x}-2(1-\cos(2\theta))(2+\cos(2\theta)) K^\flat_{p_2-p_1}\odot J^\flat_{e_x},\] 
\[h^{(1)}_b=\frac{6}{5} (2+\cos(2\theta)) K^\flat_{p_2-p_1}\odot K^\flat_{e_y}-2(1+\cos(2\theta))(2-\cos(2\theta)) K^\flat_{p_1+p_2}\odot J^\flat_{e_y},\]
\[h^{(1)}_c= K^\flat_{e_z}\odot K^\flat_{e_z},\]
\[h^{(1)}_d= K^\flat_{e_x}\odot J^\flat_{e_z}.\]
From the above discussion, $h^{(1)}_a,\ h^{(1)}_b,\ h^{(1)}_c,\ h^{(1)}_d$ are globally smooth symmetric $(0,2)$-tensors on $M$.
The first order perturbation $h^{(1)}$ is defined by
\begin{equation}\label{def.h1}
h^{(1)}=h^{(1)}_a+h^{(1)}_b+\lambda_c h^{(1)}_c+\lambda_d h^{(1)}_d,
\end{equation}
where $\lambda_c$ and $\lambda_d$ are constants to be determined later.

The second order perturbation $h^{(2)}$ consists of ten components, which we define as follows. Let
\begin{align*} 
f_{aa} 
&= -\frac{2203841 \cos(2\theta )}{1920}+\frac{5579141 \cos(4\theta )}{153600}-\frac{214597 \cos(6\theta)}{103680} \\ 
&+\frac{48211 \cos(8 \theta )}{307200}-\frac{251 \cos(10\theta)}{19200}+\frac{25 \cos(12\theta)}{165888} \\ 
&-\frac{42025}{6} \sqrt{\frac{5}{3}} \, \sum_{n=1}^\infty \frac{(-4+\sqrt{15})^n\cos(2n\theta)}{n^2}. 
\end{align*}
We define
\[h^{(2)}_{aa}=9 f_{aa}\, K^\flat_{e_x}\odot K^\flat_{e_x}.\] 
Let
\begin{align*}
f_{bb} 
&= \frac{2203841 \cos(2\theta )}{1920}+\frac{5579141 \cos(4\theta )}{153600}+\frac{214597 \cos(6\theta)}{103680} \\
&+\frac{48211 \cos(8\theta)}{307200}+\frac{251 \cos(10\theta)}{19200}+\frac{25 \cos(12\theta)}{165888} \\
&-\frac{42025}{6} \sqrt{\frac{5}{3}} \, \sum_{n=1}^\infty \frac{(4-\sqrt{15})^n \cos(2n\theta)}{n^2}.
\end{align*}
We define
\[h^{(2)}_{bb}=9 f_{bb}\, K^\flat_{e_y}\odot K^\flat_{e_y}.\] 
We further define
\[h^{(2)}_{ab}= 0,\] 
\[h^{(2)}_{cc}= g(K_{e_z},K_{e_z}) K^\flat_{e_z}\odot K^\flat_{e_z},\]
\[h^{(2)}_{ac}=\frac{1}{20}\langle p_1+p_2,e_z  \rangle(-5-14\cos(2\theta)+3\cos(4\theta))K^\flat_{e_x}\odot K^\flat_{e_z},\]
\[h^{(2)}_{bc}=\frac{1}{20}\langle p_2-p_1,e_z \rangle(-5+14\cos(2\theta)+3\cos(4\theta))K^\flat_{e_y}\odot K^\flat_{e_z},\] 
\[h^{(2)}_{dd}=\frac{1}{12} g(J_{e_x},J_{e_x}) J^\flat_{e_z}\odot J^\flat_{e_z},\] 
\begin{align*}
h^{(2)}_{ad} 
&= \frac{9}{80}\langle p_2-p_1,e_x \rangle \Big ( -\frac{406}{45} -\frac{92}{45} \cos (2 \theta )+\frac{2}{5} \cos (4 \theta ) \Big ) K^\flat_{e_x}\odot K^\flat_{e_z} \\
&+\frac{9}{2}\langle p_2-p_1,e_z  \rangle \Big ( -\frac{823}{900}+ \frac{77}{450} \cos (2 \theta )+\frac{1}{100} \cos (4 \theta ) \Big ) K^\flat_{e_x}\odot K^\flat_{e_x},
\end{align*}
\begin{align*}
h^{(2)}_{bd}
&= \frac{9}{80}\langle p_1+p_2,e_x \rangle \Big ( -\frac{406}{45}+\frac{92}{45} \cos (2 \theta )+\frac{2}{5} \cos (4 \theta ) \Big ) K^\flat_{e_y}\odot K^\flat_{e_z}\\
&+\frac{9}{2}\langle p_1+p_2,e_z \rangle \Big ( -\frac{823}{900} -\frac{77}{450} \cos (2 \theta )+\frac{1}{100} \cos (4 \theta ) \Big ) K^\flat_{e_x}\odot K^\flat_{e_y}.
\end{align*}
We define a smooth function $G: (-1,1) \to \mathbb{R}$ by $G(0)=1$ and 
\[G(x) = \frac{\arctan(x)}{x} = \sum_{n=0}^\infty (-1)^n \, \frac{x^{2n}}{2n+1}\] 
for $x \in (-1,1) \setminus \{0\}$. Then the function 
\[G \Big ( \frac{\sin(2\theta)}{\sqrt{3}} \Big ) = \sum_{n=0}^\infty (-3)^{-n} \, \frac{\sin^{2n}(2\theta)}{2n+1} = \sum_{n=0}^\infty (-3)^{-n} \, \frac{(1-\cos^2(2\theta))^n}{2n+1}\] 
extends to a smooth function on $M$. We define
\begin{align*}
h^{(2)}_{cd} 
&=-\frac{1}{8} (\langle p_1+p_2,e_z  \rangle\langle p_2-p_1,e_x  \rangle+\langle p_1+p_2,e_x  \rangle\langle p_2-p_1,e_z  \rangle) \\ 
&\hspace{16mm} \cdot G \Big ( \frac{\sin(2\theta)}{\sqrt{3}} \Big ) K^\flat_{e_z}\odot K^\flat_{e_z}\\
&+\frac{\langle p_1+p_2,e_z  \rangle\langle p_2-p_1,e_z  \rangle}{2(4-\cos^2(2\theta))}K^\flat_{e_x}\odot K^\flat_{e_z}.
\end{align*}
From the above discussion, $h^{(2)}_{aa}$, $h^{(2)}_{bb}$, $h^{(2)}_{ab}$, $h^{(2)}_{cc}$, $h^{(2)}_{ac}$, $h^{(2)}_{bc}$, $h^{(2)}_{dd}$, $h^{(2)}_{ad}$, $h^{(2)}_{bd}$, and $h^{(2)}_{cd}$ are globally smooth symmetric $(0,2)$-tensors on $M$.
Finally, we set
\begin{align}\label{def.h2}
h^{(2)} 
&=h^{(2)}_{aa}+h^{(2)}_{bb}+2 h^{(2)}_{ab}+\lambda_c^2h^{(2)}_{cc}+2\lambda_c h^{(2)}_{ac}+2\lambda_c h^{(2)}_{bc} \notag \\ 
&+\lambda_d^2 h^{(2)}_{dd}+2\lambda_d h^{(2)}_{ad}+2\lambda_d h^{(2)}_{bd}+2\lambda_c\lambda_d h^{(2)}_{cd}.
\end{align}
After these preparations, we define 
\[g_s = g + s \, h^{(1)} + s^2 \, h^{(2)}.\] 
Given a point $y \in M_{\text{\rm nondeg}}$ and a vector $z \in \mathbb{R}^4$, we consider the two-plane $\text{\rm span}\{m_3(z),m_4(z)\} \subset T_y M$, where 
\[ m_3(z)=m_3+z^1 m_1+z^3 m_2,\ m_4(z)=m_4+z^4 m_1+z^2 m_2. \] 
We define a function $v: M_{\text{\rm nondeg}} \times \mathbb{R}^4 \times [-s_0,s_0] \to \mathbb{R}$ by 
\[v(y,z,s) = R_{g_s}(m_3(z),m_4(z),m_3(z),m_4(z))\] 
for $y \in M_{\text{\rm nondeg}}$, $z \in \mathbb{R}^4$, and $s \in [-s_0,s_0]$. We may write 
\begin{equation} 
\label{Taylor.expansion.of.v}
v(y,z,s) = v^{(0)}(y,z) + s \, v^{(1)}(y,z) + s^2 \, v^{(2)}(y,z) + s^3 \, v^{(3)}(y,z) + O(s^4), 
\end{equation}
where $v^{(0)},v^{(1)},v^{(2)},v^{(3)}: M_{\text{\rm nondeg}} \times \mathbb{R}^4 \to \mathbb{R}$ are smooth functions. Clearly, $v^{(0)}(y,0) = 0$ for all $y \in M_{\text{\rm nondeg}}$.

\begin{proposition}
\label{first.order.change.in.sectional.curvature}
We have $v^{(1)}(y,0) = 0$ for all $y \in M_{\text{\rm nondeg}}$.
\end{proposition}

\textbf{Proof.} 
From the discussion in Section~\ref{sec:formula.metric.perturbation}, we have 
\[v^{(1)}(y,0)=P_0L h^{(1)}=(Lh^{(1)})(m_3,m_4,m_3,m_4).\] 
For each great circle $\Gamma$ in $S^2$, the torus $\Gamma \times \Gamma$ is totally geodesic and the vector fields $m_3$ and $m_4$ are tangential along $\Gamma \times \Gamma$. Moreover, $m_3$ and $m_4$ define parallel vector fields along $\Gamma \times \Gamma$. This implies
\begin{align*}
&(Lh^{(1)})(m_3,m_4,m_3,m_4) \\ 
&= -\frac{1}{2}m_3m_3h^{(1)}(m_4,m_4)-\frac{1}{2}m_4m_4h^{(1)}(m_3,m_3)+m_3m_4 h^{(1)}(m_3,m_4).
\end{align*}
Here $h^{(1)}(m_3,m_3)$, $h^{(1)}(m_4,m_4)$, and $h^{(1)}(m_3,m_4)$ are viewed as scalar functions, and $m_im_j$ denotes the corresponding iterated directional derivatives.

In view of \eqref{K.omega.in.mi} and \eqref{J.omega.in.mi}, we have
\[
K^\flat_\omega(m_3)=\frac{1}{3}\langle q_3,\omega \rangle,\ K^\flat_\omega(m_4)=0,\ J^\flat_\omega(m_3)=0,\ J^\flat_\omega(m_4)=-\langle q_3,\omega \rangle.
\]
In particular,
\[
K^\flat_{p_1+p_2}(m_3)=0,\ K^\flat_{p_1+p_2}(m_4)=0,\ K^\flat_{p_2-p_1}(m_3)=0,\ K^\flat_{p_2-p_1}(m_4)=0. 
\] 
This implies 
\[h^{(1)}_a(m_3,m_3) = h^{(1)}_b(m_3,m_3) = 0,\] 
\[h^{(1)}_a(m_4,m_4) = h^{(1)}_b(m_4,m_4) = 0,\] 
and 
\[h^{(1)}_a(m_3,m_4) = h^{(1)}_b(m_3,m_4) = 0.\]
Consequently, 
\[h^{(1)}(m_3,m_3) = \lambda_c \, K^\flat_{e_z}(m_3) \, K^\flat_{e_z}(m_3) + \lambda_d \, K^\flat_{e_x}(m_3) \, J^\flat_{e_z}(m_3) = \frac{\lambda_c}{9} \, \langle q_3,e_z \rangle^2,\] 
\[h^{(1)}(m_4,m_4) = \lambda_c \, K^\flat_{e_z}(m_4) \, K^\flat_{e_z}(m_4) + \lambda_d \, K^\flat_{e_x}(m_4) \, J^\flat_{e_z}(m_4) = 0,\] 
and 
\begin{align*} 
h^{(1)}(m_3,m_4) 
&= \lambda_c \, K^\flat_{e_z}(m_3) \, K^\flat_{e_z}(m_4) \\ 
&+ \frac{\lambda_d}{2} \, K^\flat_{e_x}(m_3) \, J^\flat_{e_z}(m_4) + \frac{\lambda_d}{2} \, J^\flat_{e_z}(m_3) \, K^\flat_{e_x}(m_4) \\ 
&= -\frac{\lambda_d}{6} \, \langle q_3,e_x \rangle \, \langle q_3,e_z \rangle. 
\end{align*}
Finally, the directional derivative of $q_3$ along $m_3$ vanishes, and the directional derivative of $q_3$ along $m_4$ vanishes as well. Putting these facts together, we conclude that $(Lh^{(1)})(m_3,m_4,m_3,m_4) = 0$. This finally gives $v^{(1)}(y,0)=0$. A separate verification of the identity $v^{(1)}(y,0)=0$ is provided in the companion \textsc{Mathematica} notebook. This completes the proof of Proposition \ref{first.order.change.in.sectional.curvature}. \\


In the following, we consider the functions $V^{(2)}$ and $V^{(3)}$ defined in Section~\ref{sec:formula.metric.perturbation}. Recall that $V^{(2)}$ and $V^{(3)}$ are smooth functions on $M_{\text{\rm nondeg}}$. 

We consider the 2D torus $\Sigma = \{p \in M: \langle p_1,e_z \rangle = \langle p_2,e_z\rangle = 0\}$, where $e_z$ denotes the vertical unit vector in $\mathbb{R}^3$.

\begin{proposition}
\label{second.order.change.in.sectional.curvature}
There exist positive constants $\lambda_0$ and $\delta$ such that
\[V^{(2)}(y) \geq \delta \, \Big ( 1 - \frac{\langle y_1 \wedge y_2,e_z \rangle^2}{|y_1 \wedge y_2|^2} \Big )\] 
for each point $y = (y_1,y_2) \in M_{\text{\rm nondeg}}$ provided $|\lambda_d|\leq \lambda_0$. Moreover, $V^{(2)}(y) = 0$ for each point $y \in \Sigma \cap M_{\text{\rm nondeg}}$. 
\end{proposition}

\textbf{Proof.} 
From Lemma~\ref{Vk.formula}, we have 
\[
V^{(2)}=P_0L h^{(2)}+P_0Q(h^{(1)},h^{(1)})+\frac{1}{2}r(h^{(1)})\cdot z(h^{(1)}).
\]
Here $P_0$, $L$, $Q$, $r$ and $z$ are operators defined in Section~\ref{sec:formula.metric.perturbation}.
All these operators are linear, except for $Q$, which is bilinear. 
Therefore, using \eqref{def.h1} and \eqref{def.h2}, we can expand $V^{(2)}$ into a sum of terms.
We define
\[V^{(2)}_{aa}=P_0L h^{(2)}_{aa}+P_0Q(h^{(1)}_a,h^{(1)}_a)+\frac{1}{2} r(h^{(1)}_a)\cdot z(h^{(1)}_a),\] 
\[V^{(2)}_{bb}=P_0L h^{(2)}_{bb}+P_0Q(h^{(1)}_b,h^{(1)}_b)+\frac{1}{2} r(h^{(1)}_b)\cdot z(h^{(1)}_b),\] 
\[V^{(2)}_{ab}=P_0L h^{(2)}_{ab}+P_0Q(h^{(1)}_a,h^{(1)}_b)+\frac{1}{2} r(h^{(1)}_a)\cdot z(h^{(1)}_b),\] 
\[V^{(2)}_{cc}=P_0L h^{(2)}_{cc}+P_0Q(h^{(1)}_c,h^{(1)}_c)+\frac{1}{2} r(h^{(1)}_c)\cdot z(h^{(1)}_c),\] 
\[V^{(2)}_{ac}=P_0L h^{(2)}_{ac}+P_0Q(h^{(1)}_a,h^{(1)}_c)+\frac{1}{2} r(h^{(1)}_a)\cdot z(h^{(1)}_c),\]
\[V^{(2)}_{bc}=P_0L h^{(2)}_{bc}+P_0Q(h^{(1)}_b,h^{(1)}_c)+\frac{1}{2} r(h^{(1)}_b)\cdot z(h^{(1)}_c),\]
\[V^{(2)}_{dd}=P_0L h^{(2)}_{dd}+P_0Q(h^{(1)}_d,h^{(1)}_d)+\frac{1}{2} r(h^{(1)}_d)\cdot z(h^{(1)}_d),\]
\[V^{(2)}_{ad}=P_0L h^{(2)}_{ad}+P_0Q(h^{(1)}_a,h^{(1)}_d)+\frac{1}{2} r(h^{(1)}_a)\cdot z(h^{(1)}_d),\]
\[V^{(2)}_{bd}=P_0L h^{(2)}_{bd}+P_0Q(h^{(1)}_b,h^{(1)}_d)+\frac{1}{2} r(h^{(1)}_b)\cdot z(h^{(1)}_d),\]
\[V^{(2)}_{cd}=P_0L h^{(2)}_{cd}+P_0Q(h^{(1)}_c,h^{(1)}_d)+\frac{1}{2} r(h^{(1)}_c)\cdot z(h^{(1)}_d).\]
Then we have
\begin{align*}
V^{(2)}= &V^{(2)}_{aa}+V^{(2)}_{bb}+2V^{(2)}_{ab}+\lambda_c^2 V^{(2)}_{cc}+2\lambda_c V^{(2)}_{ac}+2\lambda_c V^{(2)}_{bc}\\
&+\lambda_d^2 V^{(2)}_{dd}+2\lambda_d V^{(2)}_{ad}+2\lambda_d V^{(2)}_{bd}+2\lambda_c\lambda_d V^{(2)}_{cd}.    
\end{align*}
Verifications of the ten identities below, \eqref{V2aa.value}--\eqref{V2cd.value}, are provided in the companion \textsc{Mathematica} notebook:
\begin{equation}\label{V2aa.value}
V^{(2)}_{aa}= \Big ( -\frac{3472117}{384}+\frac{42025}{6}\sqrt{\frac{5}{3}} \Big ) \langle q_3,e_x\rangle^2,
\end{equation}
\begin{equation}\label{V2bb.value}
V^{(2)}_{bb}= \Big ( -\frac{3472117}{384}+\frac{42025}{6}\sqrt{\frac{5}{3}} \Big ) \langle q_3,e_y\rangle^2,
\end{equation}
\begin{equation}\label{V2ab.value}
V^{(2)}_{ab}=0,
\end{equation}
\begin{equation}\label{V2cc.value}
V^{(2)}_{cc}=0,  
\end{equation}
\begin{equation}\label{V2ac.value}
V^{(2)}_{ac}=0,  
\end{equation}
\begin{equation}\label{V2bc.value}
V^{(2)}_{bc}=0,  
\end{equation}
\begin{align}\label{V2dd.value}
V^{(2)}_{dd} 
&= \frac{3+4\cos(2\theta)-\cos(4\theta)}{24 (2+\cos(2\theta))^3}\langle q_1,e_z \rangle^2\langle q_3,e_x \rangle^2 \notag \\ 
&+\frac{3-4\cos(2\theta)-\cos(4\theta)}{24 (2-\cos(2\theta))^3}\langle q_2,e_z \rangle^2\langle q_3,e_x \rangle^2 \notag \\
&-\frac{11\cos(2\theta)+\cos(6\theta)}{3(7-\cos(4\theta))^2} \langle q_1,e_x \rangle\langle q_1,e_z\rangle \langle q_3,e_x \rangle\langle q_3,e_z \rangle \\
&+\frac{11\cos(2\theta)+\cos(6\theta)}{3(7-\cos(4\theta))^2} \langle q_2,e_x \rangle\langle q_2,e_z \rangle \langle q_3,e_x \rangle\langle q_3,e_z \rangle, \notag
\end{align}
\begin{equation}\label{V2ad.value}
V^{(2)}_{ad}=0,
\end{equation}
\begin{equation}\label{V2bd.value}
V^{(2)}_{bd}=0,
\end{equation}
\begin{equation}\label{V2cd.value}
V^{(2)}_{cd}=0.
\end{equation}
From the above identities, it is clear that $V^{(2)}(y)=0$ for all $y\in\Sigma\cap M_{\text{\rm nondeg}}$. Using the identities 
\[\langle q_3,e_x \rangle^2+\langle q_3,e_y \rangle^2 = 1-\langle q_3,e_z \rangle^2\] 
and 
\[\langle q_1,e_z \rangle^2+\langle q_2,e_z \rangle^2 = 1-\langle q_3,e_z \rangle^2,\] 
we obtain 
\[V^{(2)}_{aa} + V^{(2)}_{bb} = \Big ( -\frac{3472117}{384}+\frac{42025}{6}\sqrt{\frac{5}{3}} \Big ) (1-\langle q_3,e_z \rangle^2)\] 
and 
\[|V^{(2)}_{dd}| \leq C \, (|\langle q_1,e_z \rangle| + |\langle q_2,e_z \rangle|) \, |\langle q_3,e_x \rangle| \leq 2C (1-\langle q_3,e_z \rangle^2).\] 
Note that $-\frac{3472117}{384}+\frac{42025}{6}\sqrt{\frac{5}{3}}\approx 0.369$ is a positive real number. Let 
\[\delta=\frac{1}{2} \Big ( -\frac{3472117}{384}+\frac{42025}{6}\sqrt{\frac{5}{3}} \Big ).\]
Then 
\[V^{(2)} \geq \delta (1 - \langle q_3,e_z \rangle^2)\] 
for $|\lambda_d| \leq \sqrt{\frac{\delta}{2C}}$. Since $q_3 = \frac{y_1 \wedge y_2}{|y_1 \wedge y_2|}$, the assertion follows. This completes the proof of Proposition~\ref{second.order.change.in.sectional.curvature}. \\

\begin{proposition}
\label{V3.extends.smoothly}
The function $V^{(3)}|_{\Sigma \cap M_{\text{\rm nondeg}}}$ extends to a smooth function on $\Sigma$. 
\end{proposition}

\textbf{Proof.}
We consider the function $u: \mathcal{N} \times [-s_0,s_0] \to \mathbb{R}$ which assigns to a two-plane $\pi = \text{\rm span}\{\xi,\eta\} \in \mathcal{N}$ and a real number $s$ the quantity 
\begin{equation} 
\label{definition.of.u}
\frac{R_{g_s}(\xi,\eta,\xi,\eta)}{|\xi|_g^2 \, |\eta|_g^2 - \langle \xi,\eta \rangle_g^2}. 
\end{equation}
Note that the numerator in (\ref{definition.of.u}) is computed with respect to the metric $g_s$, whereas the denominator in (\ref{definition.of.u}) is computed with respect to the Cheeger-M\"uter metric $g$. It is easy to see that the quantity (\ref{definition.of.u}) depends only on the two-plane $\pi$, but is independent of the choice of the basis $\{\xi,\eta\}$. 

We may write 
\begin{equation} 
\label{Taylor.expansion.of.u}
u(\cdot,s) = u^{(0)} + s \, u^{(1)} + s^2 \, u^{(2)} + s^3 \, u^{(3)} + O(s^4), 
\end{equation}
where $u^{(0)},u^{(1)},u^{(2)},u^{(3)}: \mathcal{N} \to \mathbb{R}$ are smooth functions. Here, $u^{(0)}$ assigns to a two-plane $\pi \in \mathcal{N}$ the sectional curvature of $\pi$ with respect to the Cheeger-M\"uter metric. Clearly, $u^{(0)}|_{\mathcal{Z}} = 0$. Moreover, it follows from Corollary \ref{lower.bound.for.sectional.curvature.of.Mueter.metrics.2} that $u^{(0)} \geq \sigma \, d(\cdot,\mathcal{Z})^2$ at each point in $\mathcal{N}$, where $\sigma$ is a positive constant. Proposition \ref{first.order.change.in.sectional.curvature} implies that $u^{(1)}|_{\mathcal{Z}_{\text{\rm nondeg}}} = 0$. Since $\mathcal{Z}_{\text{\rm nondeg}}$ is dense in $\mathcal{Z}$, it follows that $u^{(1)}|_{\mathcal{Z}} = 0$.

We next define a positive smooth function $\Upsilon: M_{\text{\rm nondeg}} \times \mathbb{R}^4 \to \mathbb{R}$ and a smooth map $\Phi: M_{\text{\rm nondeg}} \times \mathbb{R}^4 \to \mathcal{N}$ as follows. Given a point $y \in M_{\text{\rm nondeg}}$ and a vector $z \in \mathbb{R}^4$, we define 
\begin{equation} 
\label{definition.of.Upsilon}
\Upsilon(y,z) = |m_3(z)|_g^2 \, |m_4(z)|_g^2 - \langle m_3(z),m_4(z) \rangle_g^2 
\end{equation}
and 
\begin{equation} 
\label{definition.of.Phi}
\Phi(y,z) = \text{\rm span}\{m_3(z),m_4(z)\} \subset T_y M 
\end{equation}
for $y \in M_{\text{\rm nondeg}}$ and $z \in \mathbb{R}^4$, where 
\[ m_3(z)=m_3+z^1 m_1+z^3 m_2,\ m_4(z)=m_4+z^4 m_1+z^2 m_2. \] 
Using the identities $|m_3|_g^2 = \frac{1}{3}$, $|m_4|_g^2 = 1$, and $\langle m_3,m_4 \rangle_g = 0$, we obtain 
\[\Upsilon(y,0) = |m_3|_g^2 \, |m_4|_g^2 - \langle m_3,m_4 \rangle_g^2 = \frac{1}{3}\]
for each $y \in M_{\text{\rm nondeg}}$. Moreover, 
\[\Phi(y,0) = \text{\rm span}\{m_3,m_4\} \in \mathcal{Z}_y\] 
for each $y \in M_{\text{\rm nondeg}}$. We next observe that 
\[v(y,z,s) = \Upsilon(y,z) \, u(\Phi(y,z),s)\]
for $y \in M_{\text{\rm nondeg}}$, $z \in \mathbb{R}^4$, and $s \in [-s_0,s_0]$. This implies 
\begin{align*}
v^{(0)}(y,z) &= \Upsilon(y,z) \, u^{(0)}(\Phi(y,z)), \\ 
v^{(1)}(y,z) &= \Upsilon(y,z) \, u^{(1)}(\Phi(y,z)), \\ 
v^{(2)}(y,z) &= \Upsilon(y,z) \, u^{(2)}(\Phi(y,z)), \\ 
v^{(3)}(y,z) &= \Upsilon(y,z) \, u^{(3)}(\Phi(y,z)) 
\end{align*}
for $y \in M_{\text{\rm nondeg}}$ and $z \in \mathbb{R}^4$. 

After these preparations, we apply Theorem \ref{smooth.extension}. After identifying $\mathcal{Z}_{\rm nondeg}$ with $M_{\text{\rm nondeg}}$, we may view $\Upsilon$ as a positive smooth function on $\mathcal{Z}_{\text{\rm nondeg}} \times \mathbb{R}^4$, and we may view $\Phi$ as a smooth map from $\mathcal{Z}_{\text{\rm nondeg}} \times \mathbb{R}^4$ to $\mathcal{N}$. Moreover, we may view $V^{(2)}$ and $V^{(3)}$ as smooth functions on $\mathcal{Z}_{\text{\rm nondeg}}$. If we apply Theorem \ref{smooth.extension} with $\Omega = \mathcal{Z}_{\text{\rm nondeg}}$, the assertion follows. This completes the proof of Proposition \ref{V3.extends.smoothly}. \\

The verification of the following lemma is provided in the companion \textsc{Mathematica} notebook.

\begin{lemma}\label{za.zb.zc.vanish}
Let $z(\cdot)$ be the differential operator defined in Section~\ref{sec:formula.metric.perturbation}. 
Then $z(h^{(1)}_a)=z(h^{(1)}_b)=z(h^{(1)}_c)=0$ along $\Sigma$.
\end{lemma}

\begin{proposition}
\label{third.order.change.in.sectional.curvature.integral.version}
There is an open dense subset of $\mathbb{R}^2$ with the property that $\int_{\Sigma \cap M_{\text{\rm nondeg}}} V^{(3)} \neq 0$ whenever $(\lambda_c,\lambda_d)$ belongs to this subset. 
\end{proposition}

\textbf{Proof.} 
From \eqref{def.h1}, \eqref{def.h2}, and Lemma~\ref{Vk.formula}, $\int_{\Sigma \cap M_{\text{\rm nondeg}}} V^{(3)}$ is a polynomial in $\lambda_c$ and $\lambda_d$, denoted by $\mathcal{P}(\lambda_c,\lambda_d)$.
Therefore, it suffices to show that the coefficient of $\lambda_c\lambda_d^2$ in $\mathcal{P}(\lambda_c,\lambda_d)$ is non-zero. 

From Lemma~\ref{za.zb.zc.vanish}, $z(h^{(1)})=\lambda_d z(h^{(1)}_d)$ along $\Sigma$.
It follows from Lemma~\ref{Vk.formula} that the coefficient function of $\lambda_c\lambda_d^2$ in $V^{(3)}|_{\Sigma}$ is given by
\begin{align*}
&2 P_0Q(h^{(1)}_c,h^{(2)}_{dd}) \\
+&P_0C( h^{(1)}_c, h^{(1)}_d, h^{(1)}_d )+P_0C( h^{(1)}_d, h^{(1)}_c, h^{(1)}_d )+P_0C( h^{(1)}_d, h^{(1)}_d, h^{(1)}_c )\\
+&4 P_0Q(h^{(1)}_d,h^{(2)}_{cd})+2 r(h^{(2)}_{cd})\cdot z(h^{(1)}_d) \\
+&2 P_1Q(h^{(1)}_c, h^{(1)}_d)\cdot z(h^{(1)}_d)+ \frac{1}{2} z(h^{(1)}_d)^T \cdot P_2 Lh^{(1)}_c \cdot z(h^{(1)}_d).
\end{align*}
In the companion \textsc{Mathematica} notebook, we show that
\[2 P_0Q(h^{(1)}_c,h^{(2)}_{dd})=0\] 
and 
\[P_0C( h^{(1)}_c, h^{(1)}_d, h^{(1)}_d )+P_0C( h^{(1)}_d, h^{(1)}_c, h^{(1)}_d )+P_0C( h^{(1)}_d, h^{(1)}_d, h^{(1)}_c)=0\] 
at each point on $\Sigma$. In the companion \textsc{Mathematica} notebook, we further show that
\[\int_{\Sigma \cap M_{\text{\rm nondeg}}} \big( 4 P_0Q( h^{(2)}_{cd}, h^{(1)}_d )+2 r(h^{(2)}_{cd})\cdot z(h^{(1)}_d) \big)=0\] 
and 
\[\int_{\Sigma \cap M_{\text{\rm nondeg}}} \Big ( 2 P_1Q(h^{(1)}_c, h^{(1)}_d)\cdot z(h^{(1)}_d)+\frac{1}{2} z(h^{(1)}_d)^T \cdot P_2 Lh^{(1)}_c \cdot z(h^{(1)}_d) \Big ) =\frac{\pi^2}{18\sqrt{3}}.\] 
Combining the above, the coefficient of $\lambda_c\lambda_d^2$ in $\mathcal{P}(\lambda_c,\lambda_d)$ is $\frac{\pi^2}{18\sqrt{3}}$. 
This completes the proof of Proposition~\ref{third.order.change.in.sectional.curvature.integral.version}.\\

\section{A third order deformation of the Cheeger-M\"uter metric}

\label{third.order.deformation}

From now on, we assume that $\lambda_c$ and $\lambda_d$ are chosen so that the conclusions of Proposition \ref{second.order.change.in.sectional.curvature} and Proposition \ref{third.order.change.in.sectional.curvature.integral.version} hold. It follows from Proposition \ref{V3.extends.smoothly} that $V^{(3)}|_{\Sigma \cap M_{\text{\rm nondeg}}}$ extends to a smooth function on $\Sigma$. Moreover, $\int_{\Sigma \cap M_{\text{\rm nondeg}}} V^{(3)} \neq 0$ by Proposition \ref{third.order.change.in.sectional.curvature.integral.version}. Consequently, we can find a real number $\mu \neq 0$ and a smooth function $\chi: \Sigma \to \mathbb{R}$ such that 
\begin{equation} 
\label{pde.for.chi}
V^{(3)} = \mu + \Delta_\Sigma \chi 
\end{equation}
at each point on $\Sigma \cap M_{\text{\rm nondeg}}$. Here, $\Sigma$ is equipped with the metric induced by the metric $g = \frac{1}{2} \, g_{\text{\rm CM},1}$ on $M$ and 
\[\Delta_\Sigma \chi = 3 \, m_3(m_3(\chi)) + m_4(m_4(\chi))\] 
denotes the Laplacian of the function $\chi: \Sigma \to \mathbb{R}$ with respect to the induced metric on $\Sigma$. We extend $\chi$ to a smooth function $\chi: M \to \mathbb{R}$. We define 
\[h^{(3)} = 6\chi \, g\] 
and 
\[\tilde{g}_s = g_s + s^3 \, h^{(3)} = g + s \, h^{(1)} + s^2 \, h^{(2)} + s^3 \, h^{(3)}.\]
Given a point $y \in M_{\text{\rm nondeg}}$ and a vector $z \in \mathbb{R}^4$, we consider the two-plane $\text{\rm span}\{m_3(z),m_4(z)\} \subset T_y M$, where 
\[ m_3(z)=m_3+z^1 m_1+z^3 m_2,\ m_4(z)=m_4+z^4 m_1+z^2 m_2. \] 
We define a function $\tilde{v}: M_{\text{\rm nondeg}} \times \mathbb{R}^4 \times [-s_0,s_0] \to \mathbb{R}$ by 
\[\tilde{v}(y,z,s) = R_{\tilde{g}_s}(m_3(z),m_4(z),m_3(z),m_4(z))\] 
for $y \in M_{\text{\rm nondeg}}$, $z \in \mathbb{R}^4$, and $s \in [-s_0,s_0]$. We may write 
\[\tilde{v}(y,z,s) = v^{(0)}(y,z) + s \, v^{(1)}(y,z) + s^2 \, v^{(2)}(y,z) + s^3 \, \tilde{v}^{(3)}(y,z) + O(s^4),\] 
where $v^{(0)},v^{(1)},v^{(2)}: M_{\text{\rm nondeg}} \times \mathbb{R}^4 \to \mathbb{R}$ are defined as in (\ref{Taylor.expansion.of.v}) and $\tilde{v}^{(3)}: M_{\text{\rm nondeg}} \times \mathbb{R}^4 \to \mathbb{R}$ is a smooth function. As in Section \ref{framework}, we define 
\begin{align*} 
\tilde{V}^{(3)}(y) 
&= \tilde{v}^{(3)}(y,0) + \sum_{\alpha=1}^4 v^{(2)}_{z^\alpha}(y,0) \, \bar{z}^\alpha(y) + \frac{1}{2} \sum_{\alpha,\beta=1}^4 v^{(1)}_{z^\alpha z^\beta}(y,0) \, \bar{z}^\alpha(y) \, \bar{z}^\beta(y) \\ 
&+ \frac{1}{6} \sum_{\alpha,\beta,\gamma=1}^4 v^{(0)}_{z^\alpha z^\beta z^\gamma}(y,0) \, \bar{z}^\alpha(y) \, \bar{z}^\beta(y) \, \bar{z}^\gamma(y) 
\end{align*}
for $y \in M_{\text{\rm nondeg}}$. Note that $\tilde{V}^{(3)}$ is a smooth function defined on $M_{\text{\rm nondeg}}$. 

\begin{proposition}
\label{third.order.change.in.sectional.curvature.pointwise.version}
We have $\tilde{V}^{(3)} = \mu$ at each point on $\Sigma \cap M_{\text{\rm nondeg}}$.
\end{proposition}

\textbf{Proof.}
Note that $\tilde{g}_s = g_s + 6s^3\chi \, g$ by definition of the metric $\tilde{g}_s$. Using the identities $|m_3|_g^2 = \frac{1}{3}$, $|m_4|_g^2 = 1$, and $\langle m_3,m_4 \rangle_g = 0$, we compute 
\[\tilde{v}^{(3)}(\cdot,0) = v^{(3)}(\cdot,0) - \Delta_\Sigma \chi\] 
at each point on $\Sigma \cap M_{\text{\rm nondeg}}$. This implies 
\[\tilde{V}^{(3)} = V^{(3)} - \Delta_\Sigma \chi\] 
at each point on $\Sigma \cap M_{\text{\rm nondeg}}$. The assertion follows now from (\ref{pde.for.chi}). This completes the proof of Proposition \ref{third.order.change.in.sectional.curvature.pointwise.version}. \\

\section{Existence of a real number $s$ with the property that $\tilde{g}_s$ has positive sectional curvature}

In this section, we show that there exists a real number $s$ with the property that the metric $\tilde{g}_s$ has positive sectional curvature. In the following, we assume that $s_0 > 0$ is chosen sufficiently small. 

\begin{proposition}
\label{positivity.of.sectional.curvature}
If $\mu > 0$, then $\tilde{g}_s$ has positive sectional curvature if $s \in (0,s_0]$ is sufficiently small. If $\mu < 0$, then $\tilde{g}_s$ has positive sectional curvature if $-s \in (0,s_0]$ is sufficiently small. 
\end{proposition}

\textbf{Proof.} 
We only consider the case $\mu > 0$. The case $\mu < 0$ follows from analogous arguments. We consider the function $\tilde{u}: \mathcal{N} \times [-s_0,s_0] \to \mathbb{R}$ which assigns to a two-plane $\pi = \text{\rm span}\{\xi,\eta\} \in \mathcal{N}$ and a real number $s$ the quantity 
\begin{equation} 
\label{definition.of.tilde.u}
\frac{R_{\tilde{g}_s}(\xi,\eta,\xi,\eta)}{|\xi|_g^2 \, |\eta|_g^2 - \langle \xi,\eta \rangle_g^2}. 
\end{equation}
Note that the numerator in (\ref{definition.of.tilde.u}) is computed with respect to the metric $\tilde{g}_s$, whereas the denominator in (\ref{definition.of.tilde.u}) is computed with respect to the Cheeger-M\"uter metric $g$. It is easy to see that the quantity (\ref{definition.of.tilde.u}) depends only on the two-plane $\pi$, but is independent of the choice of the basis $\{\xi,\eta\}$. 

We may write 
\[\tilde{u}(\cdot,s) = u^{(0)} + s \, u^{(1)} + s^2 \, u^{(2)} + s^3 \, \tilde{u}^{(3)} + O(s^4),\]
where $u^{(0)},u^{(1)},u^{(2)}: \mathcal{N} \to \mathbb{R}$ are defined as in (\ref{Taylor.expansion.of.u}) and $\tilde{u}^{(3)}: \mathcal{N} \to \mathbb{R}$ is a smooth function. In particular, $u^{(0)}|_{\mathcal{Z}} = 0$ and $u^{(0)} \geq \sigma \, d(\cdot,\mathcal{Z})^2$ at each point in $\mathcal{N}$, where $\sigma$ is a positive constant. Moreover, $u^{(1)}|_{\mathcal{Z}} = 0$.

We define a positive smooth function $\Upsilon: M_{\text{\rm nondeg}} \times \mathbb{R}^4 \to \mathbb{R}$ by (\ref{definition.of.Upsilon}), and we define a smooth map $\Phi: M_{\text{\rm nondeg}} \times \mathbb{R}^4 \to \mathcal{N}$ by (\ref{definition.of.Phi}). Recall that 
\[\Upsilon(y,0) = |m_3|_g^2 \, |m_4|_g^2 - \langle m_3,m_4 \rangle_g^2 = \frac{1}{3}\] 
and 
\[\Phi(y,0) = \text{\rm span}\{m_3,m_4\} \in \mathcal{Z}_y\] 
for each $y \in M_{\text{\rm nondeg}}$. As in Section \ref{third.order.deformation}, we define a smooth function $\tilde{v}: M_{\text{\rm nondeg}} \times \mathbb{R}^4 \times [-s_0,s_0] \to \mathbb{R}$ by 
\[\tilde{v}(y,z,s) = R_{\tilde{g}_s}(m_3(z),m_4(z),m_3(z),m_4(z))\] 
for $y \in M_{\text{\rm nondeg}}$, $z \in \mathbb{R}^4$, and $s \in [-s_0,s_0]$. Clearly, 
\[\tilde{v}(y,z,s) = \Upsilon(y,z) \, \tilde{u}(\Phi(y,z),s)\]
for $y \in M_{\text{\rm nondeg}}$, $z \in \mathbb{R}^4$, and $s \in [-s_0,s_0]$. As in Section \ref{third.order.deformation}, we write 
\[\tilde{v}(y,z,s) = v^{(0)}(y,z) + s \, v^{(1)}(y,z) + s^2 \, v^{(2)}(y,z) + s^3 \, \tilde{v}^{(3)}(y,z) + O(s^4),\] 
where $v^{(0)},v^{(1)},v^{(2)},\tilde{v}^{(3)}: M_{\text{\rm nondeg}} \times \mathbb{R}^4 \to \mathbb{R}$ are smooth functions. With this understood,  
\[\tilde{v}^{(3)}(y,z) = \Upsilon(y,z) \, \tilde{u}^{(3)}(\Phi(y,z))\] 
for $y \in M_{\text{\rm nondeg}}$ and $z \in \mathbb{R}^4$.

Since $\mu > 0$, it follows from Proposition \ref{second.order.change.in.sectional.curvature} and Proposition \ref{third.order.change.in.sectional.curvature.pointwise.version} that we can find a small positive constant $\kappa$ with the following properties:
\begin{itemize} 
\item $V^{(2)}(y) \geq \kappa \, \Upsilon(y,0) \, \big ( 1 - \frac{\langle y_1 \wedge y_2,e_z \rangle^2}{|y_1 \wedge y_2|^2} \big )$ for each point $y = (y_1,y_2) \in M_{\text{\rm nondeg}}$. 
\item $V^{(2)}(y) = 0$ and $\tilde{V}^{(3)}(y) \geq \kappa \, \Upsilon(y,0)$ for each point $y \in \Sigma \cap M_{\text{\rm nondeg}}$. 
\end{itemize}

In the next step, we define a function $\rho: \mathcal{Z} \to \mathbb{R}$ as follows. Suppose that $\pi \subset T_{(p_1,p_2)} M$. For each two-plane $\pi \in \mathcal{Z}$, we can find a unit vector $b \in S^2$ such that $\langle p_1,b \rangle = \langle p_2,b \rangle = 0$ and $\pi = \text{\rm span}\{(b,0),(0,b)\}^\perp$, where the orthogonal complement is taken with respect to $g_{\text{\rm std}}$. Moreover, for a given two-plane $\pi \in \mathcal{Z}$, the vector $b \in S^2$ is uniquely determined up to sign. We define the value of the function $\rho: \mathcal{Z} \to \mathbb{R}$ at the point $\pi \in \mathcal{Z}$ to be $1-\langle b,e_z \rangle^2$. Clearly, $\rho$ is a nonnegative continuous function on $\mathcal{Z}$. The set $\{\rho=0\} \subset \mathcal{Z}$ consists of all two-planes $\pi \subset T_{(p_1,p_2)} M$ such that $\langle p_1,e_z \rangle = \langle p_2,e_z \rangle = 0$ and $\pi = \text{\rm span}\{(e_z,0),(0,e_z)\}^\perp$, where the orthogonal complement is taken with respect to $g_{\text{\rm std}}$. In other words, the set $\{\rho=0\} \subset \mathcal{Z}$ consists of all the tangent planes to the 2D torus $\Sigma$. It is easy to see that the set $\mathcal{Z}_{\text{\rm nondeg}} \cap \{\rho=0\}$ is dense in $\{\rho=0\}$.

After these preparations, we apply Theorem \ref{positivity.manifold.version}. After identifying $\mathcal{Z}_{\rm nondeg}$ with $M_{\text{\rm nondeg}}$, we may view $\Upsilon$ as a positive smooth function on $\mathcal{Z}_{\text{\rm nondeg}} \times \mathbb{R}^4$, and we may view $\Phi$ as a smooth map from $\mathcal{Z}_{\text{\rm nondeg}} \times \mathbb{R}^4$ to $\mathcal{N}$. Moreover, we may view $V^{(2)}$ and $\tilde{V}^{(3)}$ as smooth functions on $\mathcal{Z}_{\text{\rm nondeg}}$. Applying Theorem \ref{positivity.manifold.version} with $\Omega = \mathcal{Z}_{\text{\rm nondeg}}$, we conclude that $\inf_{\mathcal{N}} \tilde{u}(\cdot,s) > 0$ if $s \in (0,s_0]$ is sufficiently small. Thus, $\tilde{g}_s$ has positive sectional curvature if $s \in (0,s_0]$ is sufficiently small. This completes the proof of Proposition \ref{positivity.of.sectional.curvature}. \\

\appendix

\section{A formula for the Riemann curvature tensor of a perturbed metric}

\label{curvature.of.perturbed.metric}

In this section, we derive a formula for the Riemann curvature tensor of a perturbed metric.

\begin{proposition}
\label{curvature.tensor.of.perturbed.metric}
Let $M$ be a manifold, and let $\bar{g}$ and $g$ be two Riemannian metrics on $M$. Let $\bar{D}$ denote the Levi-Civita connection associated with $\bar{g}$, and let $\bar{R}$ denote the Riemann curvature tensor of $\bar{g}$. Then the Riemann curvature tensor of $g$ is given by 
\begin{align*} 
R_{ijkl} 
&= \frac{1}{2} \, \bar{g}^{pq} \, \bar{R}_{ijkp} \, g_{ql} + \frac{1}{2} \, \bar{g}^{pq} \, \bar{R}_{ijpl} \, g_{kq} \\ 
&+ \frac{1}{2} \, \big [ -(\bar{D}_{i,k}^2 g)_{jl} + (\bar{D}_{i,l}^2 g)_{jk} + (\bar{D}_{j,k}^2 g)_{il} - (\bar{D}_{j,l}^2 g)_{ik} \big ] \\ 
&+ \frac{1}{4} \, g^{pq} \, [(\bar{D}_j g)_{kp} + (\bar{D}_k g)_{jp} - (\bar{D}_p g)_{jk}] \, [(\bar{D}_i g)_{lq} + (\bar{D}_l g)_{iq} - (\bar{D}_q g)_{il}] \\ 
&- \frac{1}{4} \, g^{pq} \, [(\bar{D}_i g)_{kp} + (\bar{D}_k g)_{ip} - (\bar{D}_p g)_{ik}] \, [(\bar{D}_j g)_{lq} + (\bar{D}_l g)_{jq} - (\bar{D}_q g)_{jl}]. 
\end{align*} 
\end{proposition}

\textbf{Proof.} 
Let $D$ denote the Levi-Civita connection with respect to the metric $g$. We may write 
\begin{equation} 
\label{connection}
D_X Y = \bar{D}_X Y + A(X,Y), 
\end{equation}
where the tensor $A$ is given by 
\begin{equation} 
\label{definition.of.A} 
2 \, g(A(X,Y),Z) = (\bar{D}_X g)(Y,Z) + (\bar{D}_Y g)(X,Z) - (\bar{D}_Z g)(X,Y). 
\end{equation}
The identity (\ref{definition.of.A}) implies that $A(X,Y) = A(Y,X)$ and 
\begin{equation} 
\label{covariant.derivative.of.g}
(\bar{D}_X g)(Y,Z) = g(A(X,Y),Z) + g(Y,A(X,Z)). 
\end{equation}
Using (\ref{connection}), we obtain 
\begin{equation} 
\label{term.1}
D_X D_Y Z = \bar{D}_X \bar{D}_Y Z + \bar{D}_X(A(Y,Z)) + A(X,\bar{D}_Y Z) + A(X,A(Y,Z)) 
\end{equation}
and 
\begin{equation} 
\label{term.2}
D_{D_X Y} Z = \bar{D}_{\bar{D}_X Y} Z + \bar{D}_{A(X,Y)} Z + A(\bar{D}_X Y,Z) + A(A(X,Y),Z). 
\end{equation}
In the next step, we subtract (\ref{term.2}) from (\ref{term.1}). Using the identity 
\[\bar{D}_X(A(Y,Z)) = (\bar{D}_X A)(Y,Z) + A(\bar{D}_X Y,Z) + A(Y,\bar{D}_X Z),\] 
we deduce that 
\begin{align} 
\label{second.covariant.derivative.X.Y}
D_{X,Y}^2 Z 
&= \bar{D}_{X,Y}^2 Z + (\bar{D}_X A)(Y,Z) + A(X,A(Y,Z)) - A(A(X,Y),Z) \notag \\ 
&+ A(X,\bar{D}_Y Z) + A(Y,\bar{D}_X Z) - \bar{D}_{A(X,Y)} Z. 
\end{align}
Interchanging $X$ and $Y$ gives 
\begin{align} 
\label{second.covariant.derivative.Y.X}
D_{Y,X}^2 Z 
&= \bar{D}_{Y,X}^2 Z + (\bar{D}_Y A)(X,Z) + A(Y,A(X,Z)) - A(A(Y,X),Z) \notag \\ 
&+ A(Y,\bar{D}_X Z) + A(X,\bar{D}_Y Z) - \bar{D}_{A(Y,X)} Z. 
\end{align}
Subtracting (\ref{second.covariant.derivative.X.Y}) from (\ref{second.covariant.derivative.Y.X}) yields 
\begin{align*} 
-D_{X,Y}^2 Z + D_{Y,X}^2 Z 
&= -\bar{D}_{X,Y}^2 Z + \bar{D}_{Y,X}^2 Z \\ 
&- (\bar{D}_X A)(Y,Z) + (\bar{D}_Y A)(X,Z) \\ 
&- A(X,A(Y,Z)) + A(Y,A(X,Z)). 
\end{align*}
Using the identities 
\[-D_{X,Y}^2 Z + D_{Y,X}^2 Z = g^{pq} \, R(X,Y,Z,\partial_p) \, \partial_q\] 
and 
\[-\bar{D}_{X,Y}^2 Z + \bar{D}_{Y,X}^2 Z = \bar{g}^{pq} \, \bar{R}(X,Y,Z,\partial_p) \, \partial_q,\] 
we conclude that 
\begin{align*} 
R(X,Y,Z,W) 
&= \bar{g}^{pq} \, \bar{R}(X,Y,Z,\partial_p) \, g(W,\partial_q) \\ 
&- g \big ( (\bar{D}_X A)(Y,Z),W \big ) + g \big ( (\bar{D}_Y A)(X,Z),W \big ) \\ 
&- g \big ( A(X,A(Y,Z)),W \big ) + g \big ( A(Y,A(X,Z)),W \big ). 
\end{align*}
Differentiating the identity (\ref{definition.of.A}), we obtain 
\begin{align} 
\label{derivative.of.A.X.Y}
&2 \, g((\bar{D}_X A)(Y,Z),W) + 2 \, (\bar{D}_X g)(A(Y,Z),W) \notag \\ 
&= (\bar{D}_{X,Y}^2 g)(Z,W) + (\bar{D}_{X,Z}^2 g)(Y,W) - (\bar{D}_{X,W}^2 g)(Y,Z). 
\end{align}
Interchanging $X$ and $Y$ gives
\begin{align}
\label{derivative.of.A.Y.X}
&2 \, g((\bar{D}_Y A)(X,Z),W) + 2 \, (\bar{D}_Y g)(A(X,Z),W) \notag \\ 
&= (\bar{D}_{Y,X}^2 g)(Z,W) + (\bar{D}_{Y,Z}^2 g)(X,W) - (\bar{D}_{Y,W}^2 g)(X,Z). 
\end{align}
Subtracting (\ref{derivative.of.A.X.Y}) from (\ref{derivative.of.A.Y.X}) yields 
\begin{align*} 
&-2 \, g \big ( (\bar{D}_X A)(Y,Z),W \big ) + 2 \, g \big ( (\bar{D}_Y A)(X,Z),W \big ) \\ 
&- 2 \, (\bar{D}_X g)(A(Y,Z),W) + 2 \, (\bar{D}_Y g)(A(X,Z),W) \\ 
&= -(\bar{D}_{X,Y}^2 g)(Z,W) - (\bar{D}_{X,Z}^2 g)(Y,W) + (\bar{D}_{X,W}^2 g)(Y,Z) \\ 
&+ (\bar{D}_{Y,X}^2 g)(Z,W) + (\bar{D}_{Y,Z}^2 g)(X,W) - (\bar{D}_{Y,W}^2 g)(X,Z) \\ 
&= -\bar{g}^{pq} \, \bar{R}(X,Y,Z,\partial_p) \, g(W,\partial_q) - \bar{g}^{pq} \, \bar{R}(X,Y,W,\partial_p) \, g(Z,\partial_q) \\ 
&- (\bar{D}_{X,Z}^2 g)(Y,W) + (\bar{D}_{X,W}^2 g)(Y,Z) \\ 
&+ (\bar{D}_{Y,Z}^2 g)(X,W) - (\bar{D}_{Y,W}^2 g)(X,Z). 
\end{align*}
Putting these facts together, we obtain 
\begin{align*} 
&R(X,Y,Z,W) \\ 
&= \frac{1}{2} \, \bar{g}^{pq} \, \bar{R}(X,Y,Z,\partial_p) \, g(W,\partial_q) - \frac{1}{2} \, \bar{g}^{pq} \, \bar{R}(X,Y,W,\partial_p) \, g(Z,\partial_q) \\ 
&- \frac{1}{2} \, (\bar{D}_{X,Z}^2 g)(Y,W) + \frac{1}{2} \, (\bar{D}_{X,W}^2 g)(Y,Z) \\ 
&+ \frac{1}{2} \, (\bar{D}_{Y,Z}^2 g)(X,W) - \frac{1}{2} \, (\bar{D}_{Y,W}^2 g)(X,Z) \\ 
&+ (\bar{D}_X g)(A(Y,Z),W) - (\bar{D}_Y g)(A(X,Z),W) \\ 
&- g \big ( A(X,A(Y,Z)),W \big ) + g \big ( A(Y,A(X,Z)),W \big ). 
\end{align*}
Finally, the identity (\ref{covariant.derivative.of.g}) implies 
\[(\bar{D}_X g)(A(Y,Z),W) = g \big ( A(X,A(Y,Z)),W \big ) + g \big ( A(Y,Z),A(X,W) \big )\] 
and 
\[(\bar{D}_Y g)(A(X,Z),W) = g \big ( A(Y,A(X,Z)),W \big ) + g \big ( A(X,Z),A(Y,W) \big ).\] 
Thus, we conclude that 
\begin{align*} 
&R(X,Y,Z,W) \\ 
&= \frac{1}{2} \, \bar{g}^{pq} \, \bar{R}(X,Y,Z,\partial_p) \, g(W,\partial_q) - \frac{1}{2} \, \bar{g}^{pq} \, \bar{R}(X,Y,W,\partial_p) \, g(Z,\partial_q) \\ 
&- \frac{1}{2} \, (\bar{D}_{X,Z}^2 g)(Y,W) + \frac{1}{2} \, (\bar{D}_{X,W}^2 g)(Y,Z) \\ 
&+ \frac{1}{2} \, (\bar{D}_{Y,Z}^2 g)(X,W) - \frac{1}{2} \, (\bar{D}_{Y,W}^2 g)(X,Z) \\ 
&+ g \big ( A(Y,Z),A(X,W) \big ) - g \big ( A(X,Z),A(Y,W) \big ). 
\end{align*} 
This completes the proof of Proposition \ref{curvature.tensor.of.perturbed.metric}. \\

\section{A lower bound for the sectional curvature of the Cheeger-M\"uter metric}

Throughout this section, we follow the setup in Ziller's paper \cite{Ziller}. Let $M = S^2 \times S^2$, and let $g_{\text{\rm std}}$ denote the standard metric on $S^2 \times S^2$. For each $t > 0$, we denote by $g_{\text{\rm CM},t}$ the Cheeger-M\"uter metric with parameter $t$. We denote by $R_{g_{\text{\rm CM},t}}$ the Riemann curvature tensor of $g_{\text{\rm CM},t}$. As in \cite{Ziller}, we write $g_{\text{\rm CM},t}(\xi,\eta) = g_{\text{\rm std}}(C_t(\xi),\eta)$ for all vector fields $\xi$ and $\eta$, where $C_t$ is a $(1,1)$-tensor. Proposition 1.3 in \cite{Ziller} gives  
\begin{align*} 
&R_{g_{\text{\rm std}}}(\xi,\eta,\xi,\eta) + \frac{t^3}{4} \, Q([P(\xi_{\mathfrak{m}}),P(\eta_{\mathfrak{m}})],[P(\xi_{\mathfrak{m}}),P(\eta_{\mathfrak{m}})]) \\ 
&\leq R_{g_{\text{\rm CM},t}}(C_t^{-1}(\xi),C_t^{-1}(\eta),C_t^{-1}(\xi),C_t^{-1}(\eta))
\end{align*}
for all vector fields $\xi$ and $\eta$. 

In the following, we focus on the Cheeger-M\"uter metric with parameter $t=1$.

Given a point $p = (p_1,p_2) \in M$, we denote by $\mathcal{Z}_p$ the set of all two-planes $\pi \subset T_p M$ with the property that there exists a unit vector $b \in S^2$ such that $\langle p_1,b \rangle = \langle p_2,b \rangle = 0$ and $\pi = \text{\rm span}\{(b,0),(0,b)\}^\perp$, where the orthogonal complement is taken with respect to $g_{\text{\rm std}}$. If we define $\mathcal{Z} = \bigcup_{p \in M} \mathcal{Z}_p$, then $\mathcal{Z}$ is a smooth manifold of dimension $4$. Note that $\mathcal{Z}$ is the set of all two-planes that have zero sectional curvature with respect to the Cheeger-M\"uter metric $g_{\text{\rm CM},1}$.

\begin{proposition} 
\label{lower.bound.for.sectional.curvature.of.Mueter.metrics}
Let $p = (p_1,p_2)$ be a point in $M$. Suppose that $\xi,\eta \in T_p M$ are $g_{\text{\rm std}}$-orthonormal vectors satisfying 
\[R_{g_{\text{\rm CM},1}}(C_1^{-1}(\xi),C_1^{-1}(\eta),C_1^{-1}(\xi),C_1^{-1}(\eta)) < \varepsilon^2.\] 
Then there exists a unit vector $b \in S^2$ such that $|\langle p_1,b \rangle| \leq C\varepsilon$, $|\langle p_2,b \rangle| \leq C\varepsilon$, $|\langle \xi,(b,0) \rangle_{g_{\text{\rm std}}}| \leq C\varepsilon$, $|\langle \eta,(b,0) \rangle_{g_{\text{\rm std}}}| \leq C\varepsilon$, $|\langle \xi,(0,b) \rangle_{g_{\text{\rm std}}}| \leq C\varepsilon$,  $|\langle \eta,(0,b) \rangle_{g_{\text{\rm std}}}| \leq C\varepsilon$. In other words, the two-plane $\text{\rm span}\{\xi,\eta\}$ has distance at most $C\varepsilon$ from the set $\mathcal{Z}$.
\end{proposition}

\textbf{Proof.} 
It suffices to consider the case when $p_1 \neq \pm p_2$. (The general case follows by approximation.) We write $\xi = (\xi_1,\xi_2)$ and $\eta = (\eta_1,\eta_2)$, where $\xi_1,\eta_1 \in T_{p_1} S^2$ and $\xi_2,\eta_2 \in T_{p_2} S^2$. Note that 
\[R_{g_{\text{\rm std}}}(\xi,\eta,\xi,\eta) \leq R_{g_{\text{\rm CM},1}}(C_1^{-1}(\xi),C_1^{-1}(\eta),C_1^{-1}(\xi),C_1^{-1}(\eta)) \leq \varepsilon^2.\] 
This implies 
\begin{equation}
\label{a}
|\xi_1 \wedge \eta_1|^2 + |\xi_2 \wedge \eta_2|^2 \leq C\varepsilon^2. 
\end{equation}
We next observe that 
\begin{align*} 
&\frac{1}{4} \, Q([P(\xi_{\mathfrak{m}}),P(\eta_{\mathfrak{m}})],[P(\xi_{\mathfrak{m}}),P(\eta_{\mathfrak{m}})]) \\ 
&\leq R_{g_{\text{\rm CM},1}}(C_1^{-1}(\xi),C_1^{-1}(\eta),C_1^{-1}(\xi),C_1^{-1}(\eta)) \leq \varepsilon^2 
\end{align*}
at the point $(p_1,p_2)$. This implies 
\[Q(P(\xi_{\mathfrak{m}}),P(\xi_{\mathfrak{m}})) \, Q(P(\eta_{\mathfrak{m}}),P(\eta_{\mathfrak{m}})) - Q(P(\xi_{\mathfrak{m}}),P(\eta_{\mathfrak{m}}))^2 \leq C\varepsilon^2.\] 
Let $\{a_1,a_2,a_3\}$ denote the standard basis of $\mathbb{R}^3$. Let $\{\alpha_1,\alpha_2,\alpha_3\}$ denote the corresponding basis of $\mathfrak{so}(3)$, so that $\alpha_k^*|_{(p_1,p_2)} = (a_k \wedge p_1,a_k \wedge p_2) \in T_{(p_1,p_2)} M$. Then 
\begin{align*} 
&\Big ( \sum_{k=1}^3 Q(P(\xi_{\mathfrak{m}}),\alpha_k)^2 \Big ) \, \Big ( \sum_{k=1}^3 Q(P(\eta_{\mathfrak{m}}),\alpha_k)^2 \Big ) \\ 
&- \Big ( \sum_{k=1}^3 Q(P(\xi_{\mathfrak{m}}),\alpha_k) \, Q(P(\eta_{\mathfrak{m}}),\alpha_k) \Big )^2 \leq C\varepsilon^2. 
\end{align*}
Using the definition of $P$ in \cite{Ziller}, we obtain 
\begin{align*} 
Q(P(\xi_{\mathfrak{m}}),\alpha_k) 
&= \langle \xi_{\mathfrak{m}}^*,\alpha_k^* \rangle_{g_{\text{\rm std}}} \\ 
&= \langle \xi,\alpha_k^* \rangle_{g_{\text{\rm std}}} \\ 
&= \langle \xi,(a_k \wedge p_1,a_k \wedge p_2) \rangle_{g_{\text{\rm std}}} \\ 
&= \langle \xi_1,a_k \wedge p_1 \rangle + \langle \xi_2,a_k \wedge p_2 \rangle \\ 
&= \langle p_1 \wedge \xi_1 + p_2 \wedge \xi_2,a_k \rangle 
\end{align*}
and 
\begin{align*} 
Q(P(\eta_{\mathfrak{m}}),\alpha_k) 
&= \langle \eta_{\mathfrak{m}}^*,\alpha_k^* \rangle_{g_{\text{\rm std}}} \\ 
&= \langle \eta,\alpha_k^* \rangle_{g_{\text{\rm std}}} \\ 
&= \langle \eta,(a_k \wedge p_1,a_k \wedge p_2) \rangle_{g_{\text{\rm std}}} \\ 
&= \langle \eta_1,a_k \wedge p_1 \rangle + \langle \eta_2,a_k \wedge p_2 \rangle \\ 
&= \langle p_1 \wedge \eta_1 + p_2 \wedge \eta_2,a_k \rangle. 
\end{align*} 
Thus, 
\begin{align*} 
&|p_1 \wedge \xi_1 + p_2 \wedge \xi_2|^2 \, |p_1 \wedge \eta_1 + p_2 \wedge \eta_2|^2 \\ 
&- \langle p_1 \wedge \xi_1 + p_2 \wedge \xi_2,p_1 \wedge \eta_1 + p_2 \wedge \eta_2 \rangle^2 \leq C\varepsilon^2. 
\end{align*}
This inequality can be rewritten as 
\begin{equation} 
\label{b}
|(p_1 \wedge \xi_1 + p_2 \wedge \xi_2) \wedge (p_1 \wedge \eta_1 + p_2 \wedge \eta_2)| \leq C\varepsilon. 
\end{equation} 
Recall that $\xi$ and $\eta$ are $g_{\text{\rm std}}$-orthonormal. Without loss of generality, we may assume that $\langle \xi_2,\eta_2 \rangle = 0$. (Otherwise, we replace $\xi$ by $\tilde{\xi} = \cos \tau \, \xi + \sin \tau \, \eta$ and $\eta$ by $\tilde{\eta} = -\sin \tau \, \xi + \cos \tau \, \eta$ for a suitable real number $\tau$.) Since $\langle \xi,\eta \rangle_{g_{\text{\rm std}}} = 0$, it follows that $\langle \xi_1,\eta_1 \rangle = 0$. Moreover, we may assume that $|\xi_2| \leq |\eta_2|$. (Otherwise, we switch the roles of $\xi$ and $\eta$.)

Using (\ref{a}), we obtain $|\xi_1| \, |\eta_1| = |\xi_1 \wedge \eta_1| \leq C\varepsilon$ and $|\xi_2| \, |\eta_2| = |\xi_2 \wedge \eta_2| \leq C\varepsilon$. Since $|\xi_2| \leq |\eta_2|$, it follows that $|\xi_2|^2 \leq C\varepsilon$. Since $|\xi|_{g_{\text{\rm std}}} = 1$, we conclude that $|\xi_1| \geq \frac{1}{2}$, provided that $\varepsilon$ is sufficiently small. This implies $|\eta_1| \leq C\varepsilon$.  Since $|\eta|_{g_{\text{\rm std}}} = 1$, we deduce that $|\eta_2| \geq \frac{1}{2}$, provided that $\varepsilon$ is sufficiently small. This finally implies $|\xi_2| \leq C\varepsilon$. 

Using (\ref{b}) together with the estimates $|\xi_2| \leq C\varepsilon$ and $|\eta_1| \leq C\varepsilon$, we obtain 
\begin{equation}
\label{d}
|(p_1 \wedge \xi_1) \wedge (p_2 \wedge \eta_2)| \leq C\varepsilon. 
\end{equation}
Since $\langle p_1,\xi_1 \rangle = \langle p_2,\xi_2 \rangle = 0$, we have 
\[|p_1 \wedge \xi_1|^2 = |\xi_1|^2 \geq \frac{1}{4}\] 
and 
\[|p_2 \wedge \eta_2|^2 = |\eta_2|^2 \geq \frac{1}{4}.\]
Using (\ref{d}), we conclude that 
\[\min \Big \{ \Big | \frac{p_1 \wedge \xi_1}{|p_1 \wedge \xi_1|} + \frac{p_2 \wedge \eta_2}{|p_2 \wedge \eta_2|} \Big |,\Big | \frac{p_1 \wedge \xi_1}{|p_1 \wedge \xi_1|} - \frac{p_2 \wedge \eta_2}{|p_2 \wedge \eta_2|} \Big | \Big \} \leq C\varepsilon.\] 
Let 
\[b = \frac{p_1 \wedge \xi_1}{|p_1 \wedge \xi_1|}.\] 
Then $\langle b,p_1 \rangle = \langle b,\xi_1 \rangle = 0$. Moreover, since 
\[\min \Big \{ \Big | b + \frac{p_2 \wedge \eta_2}{|p_2 \wedge \eta_2|} \Big |,\Big | b - \frac{p_2 \wedge \eta_2}{|p_2 \wedge \eta_2|} \Big | \Big \} \leq C\varepsilon,\] 
we obtain $|\langle b,p_2 \rangle| \leq C\varepsilon$ and $|\langle b,\eta_2 \rangle| \leq C\varepsilon$. Finally, $|\langle b,\xi_2 \rangle| \leq |\xi_2| \leq C\varepsilon$ and $|\langle b,\eta_1 \rangle| \leq |\eta_1| \leq C\varepsilon$. This completes the proof of Proposition \ref{lower.bound.for.sectional.curvature.of.Mueter.metrics}. \\

\begin{corollary} 
\label{lower.bound.for.sectional.curvature.of.Mueter.metrics.2}
Let $p = (p_1,p_2)$ be a point in $M$. Suppose that $\xi,\eta \in T_p M$ are two vectors with the property that $C_1\xi,C_1\eta$ are $g_{\text{\rm std}}$-orthonormal and 
\[R_{g_{\text{\rm CM},1}}(\xi,\eta,\xi,\eta) < \varepsilon^2.\] 
Then the two-plane $\text{\rm span}\{\xi,\eta\}$ has distance at most $C\varepsilon$ from the set $\mathcal{Z}$.
\end{corollary}

\textbf{Proof.}
By Proposition \ref{lower.bound.for.sectional.curvature.of.Mueter.metrics}, the two-plane $\text{\rm span}\{C_1 \xi,C_1 \eta\}$ has distance at most $C\varepsilon$ from the set $\mathcal{Z}$. Moreover, $C_1$ maps each two-plane $\pi \in \mathcal{Z}$ to itself. Putting these facts together, the assertion follows. This completes the proof of Corollary \ref{lower.bound.for.sectional.curvature.of.Mueter.metrics.2}. \\


\begin{thebibliography}{9}
\bibitem{Bourguignon}
J.P.~Bourguignon, \textit{Some constructions related to H.~Hopf's conjecture on product manifolds,} Proc. Sympos. Pure Math. 27, 33--37 (1975)

\bibitem{Bourguignon-Deschamps-Sentenac}
J.P.~Bourguignon, A.~Deschamps, P.~Sentenac, \textit{Conjecture de H.~Hopf sur les produits de vari\'et\'es,} Ann. Sci. \'Ecole Norm. Sup. 5, 277--302 (1972)

\bibitem{Cheeger}
J.~Cheeger, \textit{Some examples of manifolds with nonnegative curvature,} J. Diff. Geom. 8, 623--628 (1973)

\bibitem{Hamilton}
R.~Hamilton, \textit{Three-manifolds with positive Ricci curvature,} J. Diff. Geom. 17, 255--306 (1982)

\bibitem{Hsiang-Kleiner}
W.-Y.~Hsiang and B.~Kleiner, \textit{On the topology of positively curved $4$-manifolds with symmetry,} J. Diff. Geom. 29, 615--621 (1989)

\bibitem{Mueter}
M.~M\"uter, \textit{Kr\"ummungserh\"ohende Deformationen mittels Gruppenaktionen,} Dissertation, Universit\"at M\"unster (1987)

\bibitem{Ziller}
W.~Ziller, \textit{On M.~Mueter's Ph.D. Thesis on Cheeger deformations,} arXiv:0909.0161
\end{thebibliography}
\end{document}